\documentclass[10pt]{amsart}
\ifdefined\XeTeXversion\else\pdfoutput=1\fi

\usepackage[left=1in,right=1in,top=.9in,bottom=.9in]{geometry}
\usepackage{amsmath,amssymb,amsthm}
\usepackage{microtype}
\usepackage{tikz}
\usepackage{xcolor}
\usepackage[colorlinks=true,linkcolor=blue!50!black,
  citecolor=blue!50!black,urlcolor=blue!50!black]{hyperref}

\newtheorem{theorem}{Theorem}[section]
\newtheorem{lemma}[theorem]{Lemma}
\newtheorem{proposition}[theorem]{Proposition}
\newtheorem{corollary}[theorem]{Corollary}
\theoremstyle{remark}
\newtheorem{remark}[theorem]{Remark}

\newcommand{\R}{\mathbb R}
\newcommand{\Q}{\mathbb Q}
\newcommand{\conv}{\operatorname{conv}}
\newcommand{\len}{\operatorname{len}}
\newcommand{\supp}{\operatorname{supp}}
\newcommand{\eps}{\varepsilon}
\newcommand{\seg}[2]{[#1,#2]}
\newcommand{\Lambdaarch}{\(\Lambda\)-arch}
\newcommand{\code}[1]{\texttt{#1}}

\begin{document}

\title[Bellman's problem for the golden gnomon]
{The exact solution of Bellman's lost-in-a-forest problem
for the golden gnomon}

\author{Alexander Temerev}
\address{University of Geneva, Geneva, Switzerland}
\email{alexander.temerev@unige.ch}

\author{Alessio Doria}
\address{Independent researcher}
\email{numero2405@gmail.com}

\hypersetup{
  pdftitle={The exact solution of Bellman's lost-in-a-forest problem
    for the golden gnomon},
  pdfauthor={Alexander Temerev and Alessio Doria},
  pdfsubject={An exact escape-path theorem for the golden gnomon},
  pdfkeywords={Bellman's lost-in-a-forest problem, escape path,
    golden gnomon, support function, convex geometry, calibration}
}

\thanks{Large language models were used throughout this work---Claude
Fable~5, GPT~5.6 Sol, and Claude Opus~5---to search for the extremal curve,
to draft the arguments given here, and to write the accompanying Lean
development.  The finite algebraic certificates and the independent checks
of the rational order estimates were formalized in Lean~4, and the complete
verification development accompanies the paper at
\url{https://github.com/atemerev/gnomon}.
Section~\ref{sec:boundary} states precisely which steps are machine-checked
and which are not; those checks are valid independently of how the
statements they verify were found.}

\date{July 26, 2026}

\begin{abstract}
We solve Bellman's lost-in-a-forest problem for the golden gnomon
\(G\), the isosceles triangle with equal sides \(1\) and apex angle
\(108^\circ\): the shortest curve guaranteed to reach the boundary of
\(G\) from an unknown starting position and heading is a symmetric
seven-piece path of segments, circular shoulders, and tangents, of
exactly determined length \(C=1.282676025459\ldots\).  To our
knowledge, this is the first proved exact optimum for an isosceles
triangle whose base angle is below \(45^\circ\).  The curve's
parameters come from one isolated quartic root, and \(C\) is
transcendental.  Equivalently, \(C^{-1}G\) is the smallest homothetic
golden-gnomon cover of all unit arcs.

The proof introduces a balanced support calibration: one weighted
family of escape inequalities, built on the linear relation among the
triangle's three normals, exactly saturated by the candidate---through
eighteen exact support windows---and confronting every shorter
competitor at once.  Aggregation along the normal fan compresses the
calibration to a finite zero-sum family of supported vectors;
summation by parts then bounds its total by path length whenever the
running suffix balance, the ledger, stays in the unit disk.  A local
two-gap surgery and cyclic bitonicity force a shortest hypothetical
counterexample into exactly the temporal order the ledger tolerates.
Lean~4 verifies the two finite algebraic certificate families and the
reusable discrete ledger identities and bounds.
\end{abstract}

\subjclass[2020]{Primary 52A40; Secondary 52C15, 49Q10}
\keywords{Bellman's lost-in-a-forest problem, escape path, support
function, convex geometry, calibration, golden gnomon}

\maketitle

\begin{figure}[h]
\centering
\begin{tikzpicture}[scale=7.6,line cap=round,line join=round]
  \fill[blue!3] (-.809017,0)--(.809017,0)--(0,.587785)--cycle;
  \draw[blue!35!black,thick] (-.809017,0)--(.809017,0)--(0,.587785)--cycle;
  \draw[orange!85!black,very thick]
    (-.200294,0)--(.090730,0);
  \draw[orange!85!black,very thick]
    (.090730,0)
      arc[start angle=-90,end angle=-80.477,radius=.587785];
  \draw[orange!85!black,very thick]
    (.187980,.008101)--(.258773,.019977)--(.285047,.380686)
      --(.216722,.402696);
  \draw[orange!85!black,very thick]
    (.216722,.402696)
      arc[start angle=72.145,end angle=81.668,radius=.587785];
  \draw[orange!85!black,very thick]
    (.121673,.424804)--(-.166280,.466976);
  \draw[orange!85!black,thick,->]
    (-.122,0)--(-.025,0);
  \draw[orange!85!black,thick,->]
    (.264,.091)--(.275,.242);
  \draw[orange!85!black,thick,->]
    (.049,.435)--(-.049,.449);
  \fill[orange!85!black]
    (-.200294,0) circle (.006)
    (-.166280,.466976) circle (.006)
    (.285047,.380686) circle (.006);
  \node[orange!70!black,below left=1pt] at (-.200294,0)
    {\(\scriptstyle \gamma(0)\)};
  \node[orange!70!black,above left=1pt] at (-.166280,.466976)
    {\(\scriptstyle \gamma(7)\)};
  \node[orange!70!black,right=2pt] at (.285047,.380686)
    {\(\scriptstyle K\)};
  \node[blue!45!black,below left=1pt] at (.74,0) {\(\scriptstyle G\)};
\end{tikzpicture}
\caption{The optimal route (orange) in the golden gnomon:
\(\mathcal E(G)=1.282676025459\ldots\).}
\label{fig:path}
\end{figure}
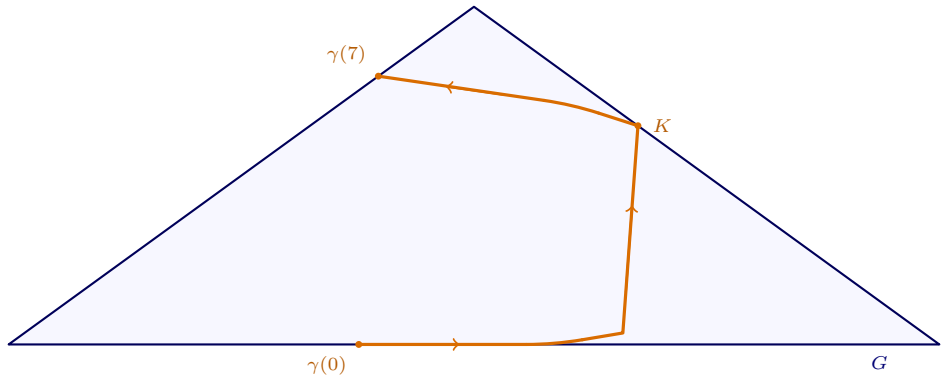

\section{Introduction}

Bellman's lost-in-a-forest problem asks for the shortest route that is
guaranteed to reach the boundary of a forest whose shape is known but in
which the starting position and heading are unknown
\cite{bellman1956,finchwetzel2004}.  For a convex forest \(K\), this is
equivalently the shortest rectifiable curve no congruent copy of which is
contained in \(\operatorname{int}K\).  Such a curve will be called an
\emph{escape path}.
Here a placement may be taken to mean a translation followed by a
rotation.  For the reflection-symmetric triangle \(G\), allowing all
Euclidean isometries gives the same notion: compose any
orientation-reversing placement with a reflection preserving \(G\).

Exact solutions are rare.  They include the disk and the classical
known-distance half-plane search variant
\cite{gross1955,isbell1957,joris1980}; the strip, whose optimum is
Zalgaller's curve---``caliper shaped,'' in the description of
\cite{gibbs2016}---of length
\(\zeta=2.27829164144\ldots\)
\cite{zalgaller1961,schaer1968,adhikaripitman1989,chan2003}; several
rectangles, sectors, and fat convex regions
\cite{gerrietspoole1974,finchwetzel2004,ward2008,panraksa2021}; the
equilateral triangle \cite{besicovitch1965,coultonmovshovich2006}; and
isosceles triangles with base angle in \([45^\circ,60^\circ]\)
\cite{movshovichwetzel2011,movshovich2012}.  We know of no previous
proved exact answer for an isosceles triangle with base angle below
\(45^\circ\); the available proposed paths were numerical
\cite{gibbs2016,ward2008}.  See \cite{movshovich2025} for a
recent survey.  A very recent formulation reduces finite
discretizations of the general problem to traveling salesman with
neighborhoods and proves convergence as the discretization is refined
\cite{deng2026}; it does not identify the exact optimizer for these
triangles.  Concurrent certified work proves that Wetzel's
\(30^\circ\!-\!60^\circ\!-\!90^\circ\) triangle, and a slightly smaller
homothet, cover every unit arc \cite{wichiramalapanraksa2026}; that result
concerns a nonisosceles cover and does not claim a sharp homothet or an
extremal escape curve.

Gibbs's computation is especially prescient here: it places the
\(36^\circ\) triangle in a ``Tunnel'' phase with the same qualitative
ingredients---two circular shoulders, tangent segments, and a central
truncation \cite{gibbs2016}.  The contribution below is to replace that
numerical conjecture by an exact construction and, more importantly, by
a global lower bound against every rectifiable competitor.

The \(45^\circ\) threshold is also a change of proof mechanism.  In the
known isosceles range the critical curve is the three-segment
Besicovitch zee, and the lower bound is organized around a finite
classification of tetral arcs \cite{movshovichwetzel2011}.  At the
golden gnomon the extremal curve has moving circular contacts and seven
pieces.  Our replacement for the tetral classification is the balanced
support calibration: normal-cone aggregation recovers a finite vector
problem, while the geometric part of the proof is reduced to forcing the
correct temporal order of those vectors.

We determine one such triangle exactly.  Put
\[
  \beta=\frac{\pi}{5},\qquad s=\sin\beta,\qquad c=\cos\beta,
\]
and let
\[
  G=\conv\{(-c,0),(c,0),(0,s)\}.
\]
The two equal sides of \(G\) have length \(1\), its base angles are
\(36^\circ\), and its apex angle is \(108^\circ\).  This triangle is the
\emph{golden gnomon}.

The angle is useful for two independent reasons.  Geometrically,
\(36^\circ\) lies strictly inside the moving-contact ``Tunnel'' regime
predicted in \cite{gibbs2016}, rather than at a phase transition or a
polygonal degeneration.  Our own numerical continuation of the whole
isosceles family, on which nothing below depends, places that regime
between a transversal crossing with the Besicovitch--Movshovich zee near
\(42.29^\circ\) and a tangential merge into the scaled Zalgaller curve
near \(27.6^\circ\), with circular contacts present throughout below
\(42.9^\circ\); the golden gnomon lies at least \(6^\circ\) from either
end, and at the upper endpoint \(45^\circ\) the branch degenerates to a
rectangle path of length \(\sqrt2\), tying the diameter of the right
isosceles triangle.  Algebraically,
\(2c=(1+\sqrt5)/2\) and \(s^2=(5-\sqrt5)/8\); the contact and calibration
equations consequently collapse to one quartic with explicit low-degree
algebraic coefficients.  The golden gnomon is therefore a natural exact
test case for the new proof mechanism, not merely a mnemonic special
angle.

For a compact convex set \(K\), write
\[
 \mathcal E(K)=\inf\{\len\gamma:\gamma\text{ is an escape path for }K\}.
\]
Our main theorem gives both the value of \(\mathcal E(G)\) and an extremal
curve.  The constant is exact: it is determined by a uniquely isolated
zero of the explicit quartic \eqref{eq:calibration-quartic}.

\begin{theorem}\label{thm:main}
Let \(a,b,\lambda\) be the exact calibration parameters defined in
Section~\ref{sec:construction}, and put
\[
  C=2sc\left(\frac{b-a}{c}+\lambda\right).
\]
Then
\[
                            \mathcal E(G)=C.
\]
The minimum is attained by the seven-piece path \(\Gamma\) in
\eqref{eq:path}.
\end{theorem}

\begin{remark}
The constant is exact and transcendental; see
Corollary~\ref{cor:transcendence}.  Numerically
\[
              C=1.282676025459048056\ldots,
\]
which is approximately \(4.22\%\) below
\(\zeta\sin\beta=1.339146227260\ldots\), the length of Zalgaller's curve
scaled to the minimal width of \(G\).  This is the classical analytic
benchmark; Gibbs's unproved ``Tunnel'' phase had already predicted that
a shorter seven-piece path should exist \cite{gibbs2016}.
\end{remark}

\begin{corollary}[sharp golden-gnomon worm cover]\label{cor:worm-cover}
The homothet
\[
                              W=C^{-1}G
\]
contains a congruent copy of every rectifiable arc of length \(1\).
No smaller positive homothet of \(G\) has this property.  Numerically, its
scale is
\[
                              C^{-1}=0.779620091240\ldots .
\]
\end{corollary}

\begin{proof}
Homogeneity gives \(\mathcal E(W)=1\).  Let \(\alpha\) be a unit arc,
translated so that \(\alpha(0)=0\), and put
\(\alpha_n=(1-1/n)\alpha\) for \(n\ge2\).  Since
\(\len\alpha_n<\mathcal E(W)\), the arc \(\alpha_n\) is not an escape
path, so there are rotations \(R_n\) and translations \(t_n\) with
\[
                         t_n+R_n\alpha_n\subset\operatorname{int}W.
\]
The rotations have a convergent subsequence.  The translations along
that subsequence are bounded because
\(t_n=t_n+R_n\alpha_n(0)\in W\), and hence have a further convergent
subsequence.  If \(R_n\to R\) and \(t_n\to t\), then, for every
parameter value \(u\),
\[
             t_n+R_n\alpha_n(u)\longrightarrow t+R\alpha(u).
\]
The left side lies in \(W\), so closedness gives
\(t+R\alpha(u)\in W\).  Thus \(W\) contains a congruent copy of the
whole arc \(\alpha\).

For sharpness, suppose \(rG\) with \(r<C^{-1}\) covered every unit arc.
It would cover \(\Gamma/C\).  Scaling that inclusion by \(C\) would put a
congruent copy of \(\Gamma\) in the smaller triangle \(CrG\), where
\(Cr<1\).  Write
\[
 G=\{(x,y):y\ge0,\quad sx+cy\le sc,\quad -sx+cy\le sc\}.
\]
If \(0<\eps<(1-Cr)s\), then \((0,\eps)+CrG\) has base margin \(\eps\)
and both oblique-side margins \((1-Cr)sc-c\eps>0\), so it lies in
\(\operatorname{int}G\).  It would contain a congruent copy of \(\Gamma\),
contrary to Theorem~\ref{thm:main}.
\end{proof}

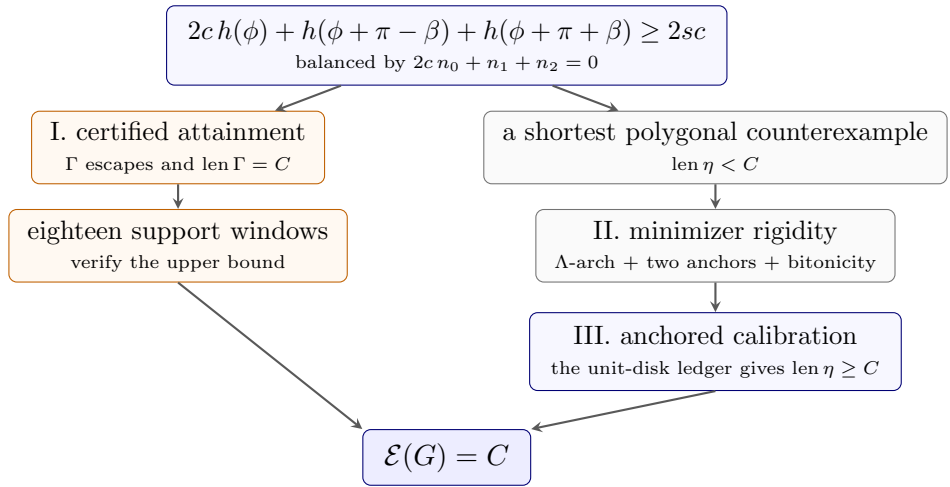
\begin{figure}[ht]
\centering
\begin{tikzpicture}[
  x=1cm,y=1cm,>=stealth,
  engine/.style={draw=blue!45!black,fill=blue!3,rounded corners=3pt,
    inner xsep=8pt,inner ysep=5pt,align=center},
  primal/.style={draw=orange!75!black,fill=orange!5,rounded corners=3pt,
    inner xsep=7pt,inner ysep=4pt,align=center},
  proof/.style={draw=black!55,fill=black!2,rounded corners=3pt,
    inner xsep=7pt,inner ysep=4pt,align=center},
  arr/.style={->,thick,black!65}
]
  \node[engine] (engine) at (0,0)
    {\(\displaystyle
      2c\,h(\phi)+h(\phi+\pi-\beta)+h(\phi+\pi+\beta)\ge2sc\)\\[-1pt]
     \scriptsize balanced by \(2c\,n_0+n_1+n_2=0\)};
  \node[primal] (curve) at (-3.55,-1.35)
    {I. certified attainment\\[-1pt]\scriptsize
     \(\Gamma\) escapes and \(\len\Gamma=C\)};
  \node[proof] (competitor) at (3.55,-1.35)
    {a shortest polygonal counterexample\\[-1pt]\scriptsize
     \(\len\eta<C\)};
  \node[primal] (windows) at (-3.55,-2.65)
    {eighteen support windows\\[-1pt]\scriptsize verify the upper bound};
  \node[proof] (rigidity) at (3.55,-2.65)
    {II. minimizer rigidity\\[-1pt]\scriptsize
     \(\Lambda\)-arch \(+\) two anchors \(+\) bitonicity};
  \node[engine] (calibration) at (3.55,-4.05)
    {III. anchored calibration\\[-1pt]\scriptsize
     the unit-disk ledger gives \(\len\eta\ge C\)};
  \node[engine,fill=blue!7,font=\large] (answer) at (0,-5.45)
    {\(\mathcal E(G)=C\)};
  \draw[arr] (engine)--(curve);
  \draw[arr] (engine)--(competitor);
  \draw[arr] (curve)--(windows);
  \draw[arr] (competitor)--(rigidity);
  \draw[arr] (rigidity)--(calibration);
  \draw[arr] (windows.south)--(answer.north west);
  \draw[arr] (calibration.south)--(answer.north east);
\end{tikzpicture}
\caption{The proof at a glance.  One balanced escape inequality is read
once on the candidate and once on a shortest polygonal counterexample.
The only geometric bridge is the
minimizer-rigidity statement.}
\label{fig:proof-map}
\end{figure}

Section~\ref{sec:proof-spine} proves the theorem from these three
statements before the technical rigidity argument is opened.  The
self-contained local surgery and the anatomy of the two finite certificate
families are then given in the appendices.  Lean verifies those certificates
and algebraic identities; compactness, hull surgery, and planar incidence
remain ordinary mathematical arguments.

\section{The extremal curve}\label{sec:construction}

Write \(u_\theta=(\cos\theta,\sin\theta)\), let \(R_\theta\) denote
counterclockwise rotation through angle \(\theta\), and put
\(\sigma(x,y)=(-x,y)\).  Let \((a,b,\lambda)\) be the exact triple
constructed from one quartic root in Section~\ref{sec:exact-data}.
Geometrically, it is characterized by the three calibration equations
\[
\begin{aligned}
0={}&-\sin a-1
 -\frac{\sin(a+\beta)-\sin(b+\beta)}{2c}+\lambda s,\\
0={}&-\cos a
 -\frac{\cos(b+\beta)-\cos(a+\beta)}{2c}+\lambda c,\\
0={}&\frac{\sin(a-\beta)-\sin(b-\beta)}{2c}-\sin b.
\end{aligned}                                                    \tag{2.1}\label{eq:calibration-system}
\]
Numerically,
\[
 \begin{aligned}
 a&=-0.238925801982\ldots,&
 b&=-0.072710623595\ldots,&
 \lambda&=1.143232127327\ldots .
 \end{aligned}
\]
The decimals are only orientation; the definition and every calculation
below use the exact triple.

Let \(O=(x,y)\) and define
\[
\begin{aligned}
 A(\theta)&=O+s u_\theta,              &
 T_1&=A(a),& T_2&=A(b),\\
 \tau&=\frac{y+s\sin a}{\cos a},       &
 K&=T_1+\tau(\sin a,-\cos a),\\
 E&=T_2+d(-\sin b,\cos b).
\end{aligned}                                                    \tag{2.2}\label{eq:geometric-points}
\]
The triple \((x,y,d)\) is now defined as the unique solution of the three
active support equations
\[
\begin{aligned}
 -2cy-sE_x+cE_y+sK_x&=0,\\
 2cx+cE_x+sE_y+cK_x&=0,\\
 cE_y+sK_x-sc&=0. 
\end{aligned}                                                    \tag{2.3}\label{eq:contacts}
\]
The first two make the two circular support windows identically sharp;
the third supplies the remaining pair of exact endpoint contacts.
To make nonsingularity inspectable, write
\(c_a=\cos a,\ s_a=\sin a,\ c_b=\cos b,\ s_b=\sin b\).
After multiplying all three equations by \(c_a>0\) to clear the
denominator in \(\tau\), the coefficient matrix in the unknowns
\((x,y,d)\) is
\[
\mathsf M=
\begin{pmatrix}
0&-c\,c_a+s\,s_a&(c\,c_b+s\,s_b)c_a\\
4c\,c_a&c\,s_a+s\,c_a&(-c\,s_b+s\,c_b)c_a\\
s\,c_a&c\,c_a+s\,s_a&c\,c_b c_a
\end{pmatrix}.                                                  \tag{2.4}\label{eq:contact-matrix}
\]
Exact rational interval evaluation gives
\[
                              \det\mathsf M>3.                  \tag{2.5}\label{eq:contact-det}
\]
Thus the system has a unique solution.  The matrix, its Cramer
numerators, and the strict determinant bound are kernel-checked in
\code{GeometryCertificate.lean}; no floating-point inversion defines
the curve.

\begin{remark}[the arc radius is forced]\label{rem:radius}
The radius \(s\) appearing in \eqref{eq:geometric-points} is dictated
rather than chosen.  Anticipating the escape sum
\[
 2c\,h(\phi)+h(\phi+\pi-\beta)+h(\phi+\pi+\beta)
\]
of (3.1), suppose a shoulder arc of radius \(r\) and centre \(O\) is the
feature exposed in the first slot, while vertices \(P_1,P_2\) are
exposed in the other two.  On such a window the sum equals
\[
 \bigl(2c\,O+R_{\beta-\pi}P_1+R_{-\beta-\pi}P_2\bigr)
              \cdot u_\phi+2c\,r .
\]
The window is identically sharp precisely when the bracketed vector
vanishes, which is what the first two equations of \eqref{eq:contacts}
assert; the surviving constant is then \(2c\,r\).  Matching the escape
threshold \(2sc\) forces
\[
                              r=s .
\]
Both tight windows of Proposition~\ref{prop:attainment} are of this
type, so the curvature of the shoulders is a consequence of the escape
threshold rather than a modelling choice.
\end{remark}

For \(P,Q\in\R^2\) and \(\theta_0,\theta_1\in\R\), let
\[
 \begin{aligned}
 \seg{P}{Q}(r)&=(1-r)P+rQ,\\
 A_{\theta_0\to\theta_1}(r)
   &=A\bigl((1-r)\theta_0+r\theta_1\bigr),
 \qquad 0\le r\le1 .
 \end{aligned}
\]
Let \(\ast\) denote oriented concatenation, each factor occupying one
unit interval.  The curve is
\[
\begin{aligned}
\gamma={}&\seg{\sigma E}{\sigma T_2}\ast
 \sigma A_{b\to a}\ast\seg{\sigma T_1}{\sigma K}\ast\seg{\sigma K}{K}\\
 &{}\ast\seg{K}{T_1}\ast A_{a\to b}\ast\seg{T_2}{E},
\qquad \Gamma=\gamma([0,7]).
\end{aligned}                                                    \tag{2.6}\label{eq:path}
\]
The endpoints of adjacent factors agree, so \(\gamma\) is a continuous
oriented path from \(\sigma E\) to \(E\).  The formula makes its reflection
symmetry manifest.

The seven successive pieces and their exact lengths are
\[
\begin{array}{c@{\qquad}c@{\qquad}c}
\text{endpoints}&\text{type}&\text{length}\\ \hline
\sigma E\to\sigma T_2&\text{segment}&d\\
\sigma T_2\to\sigma T_1&\text{circular arc}&s(b-a)\\
\sigma T_1\to\sigma K&\text{tangent segment}&\tau\\
\sigma K\to K&\text{truncating segment}&2K_x\\
K\to T_1&\text{tangent segment}&\tau\\
T_1\to T_2&\text{circular arc}&s(b-a)\\
T_2\to E&\text{segment}&d.
\end{array}                                                     \tag{2.7}\label{eq:piece-table}
\]
The two arcs have centers \(\sigma O,O\), common radius \(s\), and common
sweep \(b-a\).  Thus \eqref{eq:geometric-points} and
\eqref{eq:piece-table} give every junction, both arc centers and radii,
and every piece length exactly.

Figure~\ref{fig:path} shows the endpoint-to-endpoint member of the active
support family.  Both tips lie on sides of \(G\), the first terminal
segment runs along the base, and the third side is supported at \(K\).
The exact rigid motion realizing this placement is recorded in
Section~\ref{sec:exact-data}; the proof of universality covers every
rigid placement through its support function.

\begin{lemma}[exact length]\label{lem:length}
The signs
\[
 -\frac{\beta}{2}<a<b<0,\qquad
 \tau>0,\qquad d>0,\qquad K_x>0,\qquad \lambda>0
\]
hold, and
\[
                        \len\Gamma=C.
\]
\end{lemma}

\begin{proof}
For the first bound, \(p\ge-1/8\), \(p<0\), and
\(\arctan p>p\) give
\[
             a=2\arctan p>-\frac14>-\frac{\pi}{10}
               =-\frac{\beta}{2};
\]
here \(\arctan p>p\) follows by integrating
\((\arctan t-t)'=-t^2/(1+t^2)\) from \(p\) to \(0\), and
\(\pi>5/2\).  For the remaining signs, the certificate uses only the
following coarse boxes:
\[
\begin{gathered}
 -0.121<p<-0.119,\qquad -0.037<q<-0.036<0,\\
 -0.374<x<-0.372,\qquad 0.208<y<0.210,\qquad
 0.290<d<0.292,\\
 0.971<\cos a<0.972,\qquad -0.237<\sin a<-0.236,\qquad
 0.070<\tau<0.074,\qquad s>0.587 .
\end{gathered}                                                   \tag{2.8}\label{eq:construction-sign-boxes}
\]
In particular \(p<q<0\), so monotonicity of \(\arctan\) gives
\(a<b<0\), while \(d,\tau>0\) are
visible.  Finally
\[
\begin{aligned}
K_x
 &=x+s\cos a+\tau\sin a\\
 &>-0.374+(0.587)(0.971)-(0.074)(0.237)
  >0.178>0.
\end{aligned}
\tag{2.9}\label{eq:kx-sign}
\]
Thus every piece in \eqref{eq:piece-table} is nondegenerate with a
comfortable margin.
The contact equations \eqref{eq:contacts} and the calibration equations
\eqref{eq:calibration-system}
give the exact identity
\[
                         K_x+\tau+d=sc\lambda.                  \tag{2.10}
\]
Indeed, if
\[
 A_0=\frac{\cos b-\cos a}{2c},\qquad
 B_0=\frac{\sin b-\sin a}{2c},
\]
then direct reduction by \eqref{eq:calibration-system} gives
\[
 K_x+\tau+d-sc\lambda
   =A_0F_1+B_0F_2+\lambda F_3,
\]
where \(F_1,F_2,F_3\) are the left sides of
\eqref{eq:contacts}.  Thus the right side
vanishes.  Since \(K_x+\tau+d>0\) and \(sc>0\), this identity also gives
\(\lambda>0\).  Consequently
\[
\begin{aligned}
\len\Gamma
 &=2\{K_x+\tau+s(b-a)+d\}\\
 &=2sc\lambda+2s(b-a)
 =2sc\left(\frac{b-a}{c}+\lambda\right)=C.
\end{aligned}
\]
The two algebraic identities are also kernel-checked in Lean.
\end{proof}

\section{The support inequality and certified attainment}

For a compact convex set \(H\), write
\[
 h_H(v)=\max_{x\in H}x\cdot v,\qquad
 h_H(\theta)=h_H(u_\theta).
\]

\begin{lemma}[reflection invariance]\label{lem:reflection-invariance}
For every reflection \(S\) of the plane, a compact set \(P\) escapes
\(G\) if and only if \(S(P)\) escapes \(G\).
\end{lemma}

\begin{proof}
Fix a reflection \(J\) in the symmetry axis of \(G\), so \(J(G)=G\).
If \(P\subset A(\operatorname{int}G)\) for a translation--rotation
placement \(A\), then
\[
 S(P)\subset SA(\operatorname{int}G)
             =SAJ(\operatorname{int}G).
\]
The composition \(SAJ\) is orientation-preserving, hence is again a
translation followed by a rotation.  Thus fitting of \(P\) implies
fitting of \(S(P)\).  Applying the same argument to \(S(P)\) and using
\(S^2=\mathrm{id}\) proves the converse.
\end{proof}

\begin{lemma}[escape criterion]\label{lem:escape}
A compact set \(P\) escapes \(G\) if and only if, for
\(H=\conv P\),
\[
 2c\,h_H(\phi)+h_H(\phi+\pi-\beta)+h_H(\phi+\pi+\beta)
       \ge 2sc                                                   \tag{3.1}
\]
for every \(\phi\in\R\).
\end{lemma}

\begin{proof}
A congruent copy of \(P\) fits in \(\operatorname{int}G\) exactly when
the same is true of \(H\).  Placing \(G\) at rotation \(\phi\) and
translation \(t\) presents it as the intersection of three half-planes
with outward normals
\[
 n_0=u_\phi,\qquad n_1=u_{\phi+\pi-\beta},\qquad n_2=u_{\phi+\pi+\beta}
\]
and offsets \(0,\ 2sc,\ 0\) shifted by \(n_j\cdot t\); the reference
offsets are read from the copy of \(G\) translated so that its left base
vertex is at the origin, and any other reference point merely
reparametrizes \(t\).  Everything rests
on one identity,
\[
                 2c\,n_0+n_1+n_2=0,                              \tag{3.2}\label{eq:balance}
\]
which holds because \(n_1+n_2=2\cos\beta\,u_{\phi+\pi}=-2c\,n_0\).

Suppose first that (3.1) holds and that some placement contained \(H\),
so that \(h_H(n_j)<q_j+n_j\cdot t\) for \(j=0,1,2\), where
\((q_0,q_1,q_2)=(0,2sc,0)\).  Multiplying by the weights \(2c,1,1\) and
adding, the translation terms cancel by \eqref{eq:balance} and the
weighted offsets total \(2sc\); this contradicts (3.1).

Conversely, suppose (3.1) fails at some \(\phi\).  Put
\[
 D=2sc-\bigl(2c\,h_H(n_0)+h_H(n_1)+h_H(n_2)\bigr)>0,\qquad
 \eps=\frac{D}{2(2c+1)},
\]
and solve
\[
 n_0\cdot t=h_H(n_0)+\eps,\qquad
 n_1\cdot t=h_H(n_1)-2sc+\eps .
\]
This \(2\times2\) system has determinant
\(\sin\bigl((\phi+\pi-\beta)-\phi\bigr)=\sin\beta=s\ne0\), independent of
\(\phi\), so Cramer's rule writes \(t\) down explicitly.  The third value
is then forced by \eqref{eq:balance},
\[
 \begin{aligned}
 n_2\cdot t
 &=-2c\,(n_0\cdot t)-n_1\cdot t\\
 &=2sc-2c\,h_H(n_0)-h_H(n_1)-(2c+1)\eps\\
 &=h_H(n_2)+D-(2c+1)\eps
  =h_H(n_2)+\frac D2>h_H(n_2).
 \end{aligned}
\]
All three strict inequalities therefore hold, so
\(H\subseteq\operatorname{int}(G_\phi+t)\) and \(P\) does not escape.
\end{proof}

\begin{remark}
No separation, duality, or Farkas argument is used: the translation is
exhibited.  The identity \eqref{eq:balance} which eliminates it is the
same identity that makes the calibration measure of
Section~\ref{sec:calibration} balanced.  The translation invariance of the
problem and the calibration are thus one fact used twice, in opposite
directions.
\end{remark}

Local tangency alone is not global support: a line can touch a chain
near one point while another portion of the chain crosses to its far
side.  Monotone turning of the normal is what excludes this, and the
following principle makes the exclusion quantitative.

\begin{lemma}[one-turn support principle]\label{lem:one-turn-support}
Let \(\mathcal C\) be a closed, oriented line--circular-arc chain with
no zero-length piece.  Suppose its right-hand normal has an
unwrapped angle which is nondecreasing, including the normal-cone jumps
at vertices, and increases by exactly \(2\pi\) in one circuit.  Then
every local right-hand normal of \(\mathcal C\) is a global support normal
of the whole chain.
\end{lemma}

\begin{proof}
Fix \(P\in\mathcal C\) and a unit vector \(n\) in its closed right-hand
normal cone.  Start the cyclic traversal at \(P\), splitting the
normal-cone jump there so that its unwrapped angle \(\alpha\) begins at
\(\arg n\) and ends at \(\arg n+2\pi\).  Parametrize each open piece by
arclength.  On every such piece,
\[
 \frac d{dr}\bigl((\mathcal C(r)-P)\cdot n\bigr)
   =-|\mathcal C'(r)|\sin\bigl(\alpha(r)-\arg n\bigr).
\]
The projection is therefore nonincreasing until the normal has turned
through \(\pi\), and nondecreasing thereafter.  It is continuous at the
vertices, so jumps of the normal angle do not change this conclusion.
The projection begins and ends at zero, hence never exceeds zero.  Thus
\[
                         (Q-P)\cdot n\le0
                         \qquad(Q\in\mathcal C),
\]
which is precisely the support assertion.
\end{proof}

\begin{lemma}[support fan of the candidate]\label{lem:candidate-fan}
Let \(H=\conv\Gamma\), write
\(\bar O=\sigma O,\ \bar K=\sigma K,\ \bar E=\sigma E\), and let
\(A_R,A_L\) be the
right circular shoulder and its reflection.  With angles read modulo
\(2\pi\), one may choose the exposed feature of \(H\) as follows:
\[
\begin{array}{c|c@{\qquad}c|c}
\text{normal interval}&\text{feature}&\text{normal interval}&\text{feature}\\ \hline
{[0,\pi/2]}&E&
{[\pi/2,\pi-b]}&\bar E\\
{[\pi-b,\pi-a]}&A_L&
{[\pi-a,3\pi/2]}&\bar K\\
{[3\pi/2,2\pi+a]}&K&
{[2\pi+a,2\pi+b]}&A_R\\
{[2\pi+b,2\pi]}&E&&
\end{array}                                                       \tag{3.3}\label{eq:candidate-fan}
\]
At a common endpoint either adjacent formula may be used; the two support
values agree there.
\end{lemma}

\begin{proof}
Close \(\Gamma\) by the segment \(\seg{E}{\bar E}\).  Traverse the resulting
closed line--arc chain from \(\bar E\) along \(\Gamma\) to \(E\), and then
back to \(\bar E\).  The certified geometry enclosure gives
\(E_x>23/100\), so this closing segment is nondegenerate.  Its successive
right-hand normals, unwrapped on the real line, are
\[
\begin{gathered}
 \frac{\pi}{2},\quad
 \pi-b,\quad [\pi-b,\pi-a],\quad \pi-a,\quad
 \frac{3\pi}{2},\\
 2\pi+a,\quad[2\pi+a,2\pi+b],\quad
 2\pi+b,\quad\frac{5\pi}{2}.                                    \tag{3.4}\label{eq:normal-turn}
\end{gathered}
\]
Indeed, the two straight pieces adjacent to each circular shoulder are
tangent to it by \eqref{eq:geometric-points}; reflection gives the left
half, the middle segment is horizontal with lower outward normal, and the
closing segment is horizontal with upper outward normal.  The signs in
Lemma~\ref{lem:length} give
\[
 \frac{\pi}{2}<\pi-b<\pi-a<\frac{3\pi}{2}
 <2\pi+a<2\pi+b<\frac{5\pi}{2},
\]
and all straight pieces and circular sweeps have positive length.
Thus the right-hand normal turns monotonically through exactly one full
turn.  Lemma~\ref{lem:one-turn-support} now makes every listed local
normal a support normal of the whole closed chain; no simplicity
assumption is needed.

Let \(\mathcal C\) denote the image of this closed chain.  Its added
closing segment lies in \(\conv\Gamma\), so
\[
                         \conv\mathcal C=\conv\Gamma=H.
\]
The listed normal cones cover one full turn.  Hence, for every direction
\(n\), Lemma~\ref{lem:one-turn-support} supplies a point \(P\) on the
corresponding listed feature such that
\[
                   Q\cdot n\le P\cdot n\qquad(Q\in\mathcal C).
\]
The same inequality holds for every convex combination of points of
\(\mathcal C\); therefore \(h_H(n)=P\cdot n\).  This proves directly
that the listed feature realizes the support of \(H\), without first
having to identify the closed chain topologically with \(\partial H\).

On the two circular pieces the support formula is respectively
\(\bar O\cdot u_\theta+s\) and \(O\cdot u_\theta+s\).  Between consecutive
edge normals the intervening vertex is exposed.  Reading these vertex
cones and arc-normal intervals from \eqref{eq:normal-turn} gives exactly
\eqref{eq:candidate-fan}, including the interval for \(E\) which wraps
across \(2\pi\).
\end{proof}

\begin{proposition}[certified attainment]\label{prop:attainment}
The path \(\Gamma\) escapes \(G\) and has length \(C\).
\end{proposition}

\begin{proof}
The length identity is Lemma~\ref{lem:length}.  It remains to verify the
escape condition.

Let \(H=\conv\Gamma\), put \(\rho=2sc\), let \(B\) be the unit disk, and
set
\[
 W=2cH+R_{\beta-\pi}H+R_{\pi-\beta}H.
\]
Here \(+\) denotes Minkowski addition.  Since
\[
 h_{R_\alpha H}(\phi)=h_H(\phi-\alpha),
\]
the left side of (3.1) is exactly \(h_W(\phi)\).  Since compact convex sets satisfy
\(A\subseteq B\) if and only if \(h_A\le h_B\), the required inequality
is equivalent to the geometric inclusion \(\rho B\subseteq W\).

Use the support fan of Lemma~\ref{lem:candidate-fan}.  Superpose its three
copies in \(W\), and parameterize the relevant half-turn by the physical
normal
\[
                         \widehat u_\theta=u_{\theta+3\pi/2}.
\]
Put
\[
\begin{aligned}
(\xi_0,\ldots,\xi_9)=\bigl(&0,\ \beta,\
 {\tfrac\pi2-\beta-b},\ {\tfrac\pi2-\beta-a},\
 {\tfrac\pi2+a},\ {\tfrac\pi2+b},\\
 &{\tfrac\pi2+\beta-b},\ {\tfrac\pi2+\beta-a},\
 \pi-\beta,\ \pi\bigr).
\end{aligned}                                                    \tag{3.5}\label{eq:fan-breaks}
\]
Their strict order is elementary rather than numerical.  Starting with
\(-\beta/2<a<b<0\), the only comparisons not equal to \(a<b\) or to
positivity of \(\beta\) reduce to
\[
 b<\frac{\pi}{2}-2\beta=\frac{\pi}{10},\qquad
 2b<\beta,\qquad
 a>2\beta-\frac{\pi}{2}=-\frac{\pi}{10}=-\frac{\beta}{2}.
\]
Thus
\[
                         0=\xi_0<\xi_1<\cdots<\xi_9=\pi,
\]
and every successive gap is automatically less than \(\pi\).
On the \emph{support window} \(I_i=[\xi_i,\xi_{i+1}]\), the three
entries in
the second column below are the exposed features in the three terms of
the support sum, in the order
\(2cH,R_{\beta-\pi}H,R_{\pi-\beta}H\).  Here is an explicit recipe for the finite
certificate.  Attach to each feature \(F\) a center \(p(F)\) and a
constant \(r(F)\):
\[
\begin{array}{c|cc}
F& p(F)&r(F)\\ \hline
\text{vertex }V&V&0\\
A_R&O&s\\
A_L&\bar O&s .
\end{array}
\]
Thus the support contributed by \(F\) in direction \(u\) is
\(p(F)\cdot u+r(F)\).  If the feature triple in row \(i\) is
\((F_{i,0},F_{i,1},F_{i,2})\), define
\[
\begin{aligned}
 Q_i={}&2c\,p(F_{i,0})
       +R_{\beta-\pi}p(F_{i,1})
       +R_{\pi-\beta}p(F_{i,2}),\\
 P_i={}&R_{\pi/2}Q_i,\\
 R_i={}&2c\,r(F_{i,0})+r(F_{i,1})+r(F_{i,2})-\rho .
\end{aligned}                                                    \tag{3.6}\label{eq:window-recipe}
\]
Because \(Q_i\cdot\widehat u_\theta=P_i\cdot u_\theta\), the margin is
\[
                            m_i(\theta)=P_i\cdot u_\theta+R_i.   \tag{3.7}
\]
If \(P_i\ne0\), writing \(P_i=|P_i|u_\psi\) shows that the only possible
interior minimum is the antipodal direction
\(\theta=\psi+\pi\pmod{2\pi}\); if \(P_i=0\), the margin is constant.

\begin{center}
\small
\setlength{\tabcolsep}{5pt}
\renewcommand{\arraystretch}{1.05}
\begin{tabular}{c|c|c|c}
\(i\)&exposed feature triple&
 \(\bigl(m_i(\xi_i),m_i(\xi_{i+1})\bigr)\)&antipode witness\\ \hline
0&\(K,E,\bar E\)&\(+,+\)&L\\
1&\(K,\bar E,\bar E\)&\(+,+\)&L\\
2&\(K,\bar E,A_L\)&\(+,+\)&L\\
3&\(K,\bar E,\bar K\)&\(+,0\)&L\\
4&\(A_R,\bar E,\bar K\)&\(0,0\)&\(m_4\equiv0\)\\
5&\(E,\bar E,\bar K\)&\(0,+\)&L\\
6&\(E,A_L,\bar K\)&\(+,+\)&R\\
7&\(E,\bar K,\bar K\)&\(+,+\)&R\\
8&\(E,\bar K,K\)&\(+,0\)&L
\end{tabular}
\end{center}
Here \(+\) means \(>3/500\).  For a nonactive row,
\[
\begin{array}{ll}
\mathrm L:\ \det(u_{\xi_i},-P_i)<-17/1000,&
\mathrm R:\ \det(-P_i,u_{\xi_{i+1}})<-17/1000 .
\end{array}                                                       \tag{3.8}\label{eq:antipode-witness}
\]
Since the interval is traversed counterclockwise and has length less
than \(\pi\), either inequality puts the antipode outside \(I_i\).
Thus the endpoint column proves \(m_i\ge0\) in every row.  The middle
row is sharper: the first two equations in
\eqref{eq:contacts} give \(P_4=0\), and its circular contribution is
\(\rho\), hence the margin constant is \(R_4=0\); it is identically
tangent to \(\rho B\).  Finally,
\(\sigma H=H\) and
\(\sigma R_\theta=R_{-\theta}\sigma\), so reflection fixes \(2cH\) and
exchanges the two rotated summands of \(W\).  Thus \(\sigma W=W\), and
for the parameter \(\widehat u_\theta\) above, reflection sends
\(\theta\) to \(-\theta\).  The reflected half-turn therefore supplies
the other nine windows.  Hence
\(\rho B\subseteq W\).

Exact rational interval arithmetic in \code{EscapeWindows.lean} checks
the fan order and every strict entry; the zeros are algebraic identities.
No angular sampling is used.
\end{proof}

\section{The anchored calibration}\label{sec:calibration}

We next turn (3.1) into a sharp lower bound.  As with calibrations
elsewhere in geometry, one fixed object---the source measure
below---pairs against every competitor at once and is exactly saturated
by the candidate.  Let
\(\tau_\delta(\phi)=\phi+\delta\) on the circle of angles, and define the
positive source measure
\[
\nu=\frac{\mathbf1_{[a,b]}}{2c}\,d\phi
    +\frac{\mathbf1_{[\pi-b,\pi-a]}}{2c}\,d\phi
    +\lambda\,\delta_{\pi/2}.
\]
Fold it through the three normals of the triangle---folding is
bookkeeping, each source angle being pushed through the three normal
shifts and the results added:
\[
 \mu=2c\,\nu+(\tau_{\pi-\beta})_\#\nu
                 +(\tau_{\pi+\beta})_\#\nu .                    \tag{4.1}\label{eq:fold}
\]
Thus integrating (3.1) against \(\nu\) is exactly integration of the
support function against \(\mu\).  More importantly, the balance is now
visible pointwise:
\[
 \int u_\theta\,d\mu(\theta)
 =\int\!\left(2c\,u_\phi+u_{\phi+\pi-\beta}
                    +u_{\phi+\pi+\beta}\right)d\nu(\phi)=0,
 \qquad
 2sc\,\nu(S^1)=2sc\left(\frac{b-a}{c}+\lambda\right)=C.         \tag{4.2}
\]
The first equality is precisely \eqref{eq:balance}; the second integrates
the right side of (3.1).  In the total \(C=2s(b-a)+2sc\lambda\), the
continuous mass accounts for the candidate's two circular shoulders and
the atom for its total straight length in \eqref{eq:piece-table}.
Call the running suffix total of a finite vector family---the mass
still assigned at or after a given time---its \emph{ledger};
Lemma~\ref{lem:ledger} will keep every ledger state of \(\mu\) in the
closed unit disk.  The parameters \((a,b,\lambda)\) thus have two jobs:
their mass identity (4.2) makes the calibrated total equal \(C\), while
the three equations of \eqref{eq:calibration-system} place the critical
ledger states on the unit circle---they are its unit-circle jump
conditions.

For a path \(\eta\) with hull \(H\), write
\[
                         I_\mu(\eta)=\int h_H(\theta)\,d\mu(\theta).
\]
The logical target of the section is the sandwich
\[
                         C\le I_\mu(\eta)\le\len\eta.
\]
The first inequality is obtained simply by integrating the valid escape
inequality (3.1) with the positive weights \(\nu\).  The second is a
different statement: it follows from balance, finite normal-cone
aggregation, and the temporal unit-disk ledger.  In particular, no
identity \(I_\mu(\eta)=C\) is assumed; a general escaping path may have
a larger support total.

Translate \(\nu\) by any common angle and fold it as in
\eqref{eq:fold}; this rotates every vector of \(\mu\) by that angle.
Because (3.1) holds at every source angle, both identities above and the
integrated lower bound remain unchanged.  We shall use this freedom to
align the atom \(2c\lambda u_{\pi/2}\) with the outward normal of a
chosen supporting gap.  In the rigidity argument that gap is normalized
as an upper support, so the displayed, unrotated ledger applies
literally.

For clarity, the resulting ledger can be written explicitly.  Put
\[
\begin{aligned}
 I_R&=(\sin b-\sin a,\ \cos a-\cos b),\\
 I_L&=(-(\sin b-\sin a),\ \cos a-\cos b),\\
 f&=\frac1{2c},\qquad Q_\pm=R_{\pi\pm\beta}.
\end{aligned}
\]
The consecutive vector blocks are
\[
\begin{array}{c@{\;}c@{\;}c@{\;}c@{\;}c}
\frac12Z_0,&R_1,&L_0,&R_2,&Z_1,\\
Z_2,&L_1,&R_0,&L_2,&\frac12Z_0,
\end{array}                                                       \tag{4.3}
\]
where
\[
\begin{array}{lll}
\frac12Z_0=(0,\lambda c),&
R_1=fQ_-I_R,&L_0=I_L,\\
R_2=fQ_+I_R,&
Z_1=(-\lambda s,-\lambda c),&
Z_2=(\lambda s,-\lambda c),\\
L_1=fQ_-I_L,&R_0=I_R,&L_2=fQ_+I_L.
\end{array}                                                       \tag{4.4}\label{eq:ledger-blocks}
\]

The following elementary form of the calibration principle is the
bridge from support geometry to path length.

Because a shortest competitor may be taken polygonal
(Proposition~\ref{prop:standard}), only componentwise integration over
finitely many intervals is needed.  The following observation is the
precise reason that an interval of support directions may be compressed
to one vector.

\begin{lemma}[normal-cone aggregation]\label{lem:normal-aggregation}
Let \(H\subset\R^2\) be compact and convex, let \(P\in H\), and let
\(\omega\) be a finite positive measure on a set \(J\) of directions.
Suppose
\[
                         P\cdot u_\theta=h_H(\theta)
       \qquad\text{for \(\omega\)-almost every \(\theta\in J\)}.
\]
If
\[
                              v=\int_Ju_\theta\,d\omega(\theta),
\]
then \(P\) also supports \(H\) in direction \(v\), and
\[
             h_H(v)=P\cdot v=\int_Jh_H(\theta)\,d\omega(\theta).
                                                                  \tag{4.5}\label{eq:aggregate}
\]
\end{lemma}

\begin{proof}
For every \(x\in H\), positivity of \(\omega\) gives
\[
 (P-x)\cdot v
   =\int_J (P-x)\cdot u_\theta\,d\omega(\theta)
   =\int_J\bigl(h_H(\theta)-x\cdot u_\theta\bigr)\,d\omega(\theta)
   \ge0.
\]
Thus \(P\cdot v=\max_{x\in H}x\cdot v=h_H(v)\).  Linearity of the
integral gives the remaining equality in \eqref{eq:aggregate}.
\end{proof}

Here is the finite reduction in full.  The normal fan of a polygon has
finitely many cones.  Pull their boundary directions back through the
three shifts in \eqref{eq:fold}.  Cutting at the resulting points
decomposes the continuous part of \(\nu\) into finitely many pairwise
disjoint Borel cells; assign each cut point to either adjacent cell,
which is immaterial because those points have zero Lebesgue mass.
Treat the atom as one additional cell.  On every continuous cell \(J\),
choose vertices that support the three shifted direction families for
\(\nu\)-almost every source angle; at the atomic cell choose any vertex
of each exposed face.  Put
\[
\begin{aligned}
 v_{J,0}&=2c\int_Ju_\phi\,d\nu(\phi),\\
 v_{J,1}&=\int_Ju_{\phi+\pi-\beta}\,d\nu(\phi),&
 v_{J,2}&=\int_Ju_{\phi+\pi+\beta}\,d\nu(\phi).
\end{aligned}                                                     \tag{4.6}\label{eq:chunk-vectors}
\]
At a fan boundary either adjacent vertex may be chosen; this changes a
null set.  An atomic vector at such a direction may be assigned to
either endpoint of the exposed edge, or split between them after the
aggregation.  Lemma~\ref{lem:normal-aggregation}
applied to the three labelled families gives
\[
\begin{aligned}
 h_H(v_{J,0})+h_H(v_{J,1})+h_H(v_{J,2})
  =\int_J\!\bigl(2c\,h_H(\phi)
  +h_H(\phi+\pi-\beta)+h_H(\phi+\pi+\beta)\bigr)\,d\nu(\phi).
\end{aligned}
                                                                    \tag{4.7}\label{eq:chunk-support}
\]
Moreover, the balance identity holds on each cell, not merely after all
cells are added:
\[
 v_{J,0}+v_{J,1}+v_{J,2}
 =\int_J\!\bigl(2c\,u_\phi+u_{\phi+\pi-\beta}
                         +u_{\phi+\pi+\beta}\bigr)\,d\nu(\phi)=0.
                                                                    \tag{4.8}\label{eq:chunk-balance}
\]
Thus the support integral is exactly a finite balanced list of supported
vectors.  No measurable selection or interchange of two integrals is
being suppressed; what remains is Abel summation.

\begin{lemma}[directed support calibration]\label{lem:calibration}
Let \(P_0,\dots,P_N\) be the vertices of a polygonal path in temporal
order and let \(v_1,\dots,v_m\in\R^2\) satisfy \(\sum_jv_j=0\).  For each
\(j\) let \(\iota(j)\in\{0,\dots,N\}\) be an index with
\[
 v_j\cdot P_{\iota(j)}=h_{\conv P}(v_j),
\]
and put
\[
             R_i=\sum_{j:\ \iota(j)\ge i}v_j .
\]
If \(|R_i|\le1\) for \(i=1,\dots,N\), then
\[
 \sum_j h_{\conv P}(v_j)\ \le\ \sum_{i=0}^{N-1}|P_{i+1}-P_i| .
                                                                 \tag{4.9}\label{eq:abel}
\]
\end{lemma}

\begin{proof}
Write \(P_{\iota(j)}=P_0+\sum_{i<\iota(j)}(P_{i+1}-P_i)\).  The \(P_0\)
terms cancel because \(\sum_jv_j=0\), and exchanging the two finite sums
gives
\[
 \sum_j v_j\cdot P_{\iota(j)}
   =\sum_{i=0}^{N-1}R_{i+1}\cdot(P_{i+1}-P_i)
   \le\sum_{i=0}^{N-1}|R_{i+1}|\,|P_{i+1}-P_i| .
\]
Now apply \(|R_{i+1}|\le1\).
\end{proof}

Applied to the vectors \eqref{eq:chunk-vectors}, the left side of
\eqref{eq:abel} is exactly the calibrated total by
\eqref{eq:chunk-support}; their total sum is zero by
\eqref{eq:chunk-balance}.  The integrated escape inequalities bound that
same total below by \(C\), so a unit-disk suffix bound implies
\(\len P\ge C\).  Splitting an atom between two contacts realizing the
same support value merely replaces one \(v_j\) by two parallel vectors;
this is used for the two endpoints of the distinguished gap.

\begin{remark}
Stating Abel summation directly for a continuum of contacts would require
a measure on \(S^1\times[0,L]\), a Fubini identity, and a
measurable-selection argument.  Lemma~\ref{lem:normal-aggregation} shows
why none of that is needed for a polygon: positivity preserves a normal
cone, and the finite normal fan supplies the chunks.
\end{remark}

The data produced by the finite reduction---the source cells, their
three aggregated vector pieces, and the supporting vertex chosen for
each piece---form a \emph{chunked support marking} of the path.
Call a chunked support marking \emph{anchored} if its vector pieces occur
in the cyclic order (4.3), except that a prefix of
\((R_1,L_0,R_2)\) may occur before the first half of \(Z_0\), and a
suffix of \((L_1,R_0,L_2)\) may occur after the last half of \(Z_0\).
Subdivision of a displayed block is harmless, and an atom may be split
between two contacts with the same support value.  This definition
isolates exactly the order information used by the lower bound.

\begin{lemma}[anchored ledger bound]\label{lem:ledger}
Traverse the vector blocks in their temporal order, and let \(R(t)\)
denote the vector mass still assigned at or after time \(t\); at polygon
vertices these are the suffix states \(R_i\) of
Lemma~\ref{lem:calibration}.
For every anchored marking, \(|R(t)|\le1\).
\end{lemma}

\begin{proof}
Put \(U_t^\pm=(-\sin t,\pm\cos t)\) and \(e=(1,0)\).  Since the total
vector mass is zero, the state after removing the first half of \(Z_0\)
is
\[
                         X=-\frac12Z_0=(0,-\lambda c).
\]
Let \(Y\) be the state after \(R_2\), and write
\(Y^*=(Y_x,-Y_y)\).  The ten-block ledger is the following symmetric
walk.  An arrow labelled \(V\) means that \(V\) has just been removed
from the remaining mass, so the state changes by \(-V\):
\[
\begin{aligned}
0&\xrightarrow{\frac12Z_0}X
 \xrightarrow[\nearrow]{R_1}U_b^-
 \xrightarrow[\partial B]{L_0}U_a^-
 \xrightarrow[\searrow]{R_2}Y
 \xrightarrow{Z_1}e,\\
e&\xrightarrow{Z_2}Y^*
 \xrightarrow[\nearrow]{L_1}U_a^+
 \xrightarrow[\partial B]{R_0}U_b^+
 \xrightarrow[\searrow]{L_2}-X
 \xrightarrow{\frac12Z_0}0 .
\end{aligned}                                                   \tag{4.10}\label{eq:ledger-walk}
\]
Here \(\nearrow,\searrow\) mean that \(|R|^2\) is respectively strictly
increasing and decreasing along that block, while \(\partial B\) means
that the whole block lies on the unit circle.  Indeed, the two latter
states are exactly
\[
                         R_{L_0}(t)=U_t^-,
            \qquad       R_{R_0}(t)=U_t^+ .                    \tag{4.11}\label{eq:unit-ledger-blocks}
\]
The five unit junctions require no additional numerical check.  Direct
substitution in \eqref{eq:ledger-blocks} gives
\[
\begin{gathered}
 X-R_1=U_b^-,
 \qquad U_b^--L_0=U_a^-,
 \qquad U_a^--R_2-Z_1=e,\\
 e-Z_2-L_1=U_a^+,
 \qquad U_a^+-R_0=U_b^+ .
\end{gathered}                                                    \tag{4.12}\label{eq:unit-ledger-junctions}
\]
After the angle-addition formulas, the two coordinates of the first
identity are respectively the third and second equations of
\eqref{eq:calibration-system}; the two coordinates of the middle
identity are the first and second.  The fourth identity is the reflection
of the middle one, and the second and fifth follow immediately from
\(L_0=I_L\) and \(R_0=I_R\).  Thus every displayed junction lies on the
unit circle for a visible reason.  The two non-unit junctions are just as
transparent.  From the definition of \(X\), and from
\(Y-Z_1=e\) with \(Z_1=(-\lambda s,-\lambda c)\),
\[
\begin{gathered}
 X=(0,-\lambda c),\qquad
 Y=(1-\lambda s,-\lambda c),\\
 |X|^2=(\lambda c)^2,\qquad
 1-|Y|^2=\lambda(2s-\lambda).                                 \tag{4.13}\label{eq:strict-ledger-states}
\end{gathered}
\]
The coarse certified scalar bounds
\[
 0<\lambda<1.144,\qquad c<0.81,\qquad s>0.587
\]
give \(\lambda c<1.144\cdot0.81<1\) and
\(\lambda<1.144<1.174<2s\).  Hence both \(X\) and \(Y\) lie strictly
inside the unit disk.  This replaces a two-dimensional interval check
by two scalar comparisons with large margins.
The radial signs also have a short analytic proof.  Along \(R_1\) and
\(R_2\), respectively, the states are
\[
 S_1(t)=U_b^-+fQ_-\int_t^b u_\theta\,d\theta,\qquad
 S_2(t)=U_a^--fQ_+\int_a^t u_\theta\,d\theta .
\]
Rotation invariance of the dot product and one differentiation give
\[
\begin{aligned}
 \frac d{dt}|S_1(t)|^2
  &=-2f\{\sin(b+t-\beta)+f\sin(b-t)\},\\
 \frac d{dt}|S_2(t)|^2
  &=-2f\{\sin(a+\beta+t)-f\sin(t-a)\}.
\end{aligned}                                                   \tag{4.14}\label{eq:radial-derivatives}
\]
For \(a\le t\le b\), Lemma~\ref{lem:length} implies
\[
\begin{gathered}
 -2\beta<b+t-\beta<-\beta,\qquad
 0\le b-t<\frac{\beta}{2},\\
 0\le t-a<a+\beta+t<\beta .
\end{gathered}
\]
Since \(2\beta<\pi/2\), sine is increasing on \([0,2\beta]\); and
\(0<f<1\).  Thus
\[
 \sin(b+t-\beta)<-s,\qquad
 0\le f\sin(b-t)<\sin(\beta/2)<s,
\]
so the first brace in \eqref{eq:radial-derivatives} is negative.
Likewise
\[
 0\le f\sin(t-a)\le\sin(t-a)<\sin(a+\beta+t),
\]
so the second brace is positive.  Thus the two signs are \(+,-\).
Reflection gives the second line of \eqref{eq:ledger-walk} and the other
two signs \(+,-\).  Consequently every continuous block stays in the
unit disk.  The \(Z_j\) are atomic jumps; if an atom is split, its
intermediate states lie on the chord between its two displayed endpoints
and are safe by convexity of the disk.  This proves the bound for the
displayed order.

It remains only to move the allowed prefix and suffix.  The scalar mass
of either whole side sector is
\[
 (b-a)\left(1+\frac1c\right)
\]
because its three blocks have scalar densities
\(1/(2c),1,1/(2c)\) over an interval of length \(b-a\).
This mass is less than \(47/125<1\).  Suppose a prefix
\(\mathcal P\) is moved
before the first half-atom.  While \(\mathcal P\) is removed, the ledger
is the negative of one of its partial vector sums, so the triangle
inequality bounds its norm by the scalar mass of the whole sector,
hence by \(47/125\).  Let \(p\) be the total vector of \(\mathcal P\).
The state just before the half-atom is \(-p\); the state just after it is
the same state reached in the original walk after removing
\(\frac12Z_0\) and \(\mathcal P\), because vector addition commutes.
Both endpoints lie in the unit disk.  If the half-atom is split, all
intermediate states lie on their chord and are safe by convexity.
Thereafter every state is literally a state of
\eqref{eq:ledger-walk}.

A suffix \(\mathcal S\) moved after the final half-atom is treated by
reading this argument backwards, and it is worth making the reversal
explicit, because the naive triangle inequality is insufficient there:
just before the final half-atom the remaining mass is
\(\frac12Z_0+\sum\mathcal S\), and \(\lambda c+\frac{47}{125}>1\).
Reverse the temporal order of the whole marking.  Because the total mass
is zero, reversal negates every suffix state, so no norm changes; and
the reversed block order
\[
 \frac12Z_0,\ L_2,\ R_0,\ L_1,\ Z_2,\ Z_1,\ R_2,\ L_0,\ R_1,\ \frac12Z_0
\]
is exactly the image of \eqref{eq:ledger-walk} under the reflection
\(\sigma\), which fixes \(Z_0\) and exchanges
\(R_1\leftrightarrow L_2\), \(L_0\leftrightarrow R_0\),
\(R_2\leftrightarrow L_1\), \(Z_1\leftrightarrow Z_2\).  Since
\(\mathcal S\) is a normal-fan suffix of \((L_1,R_0,L_2)\), its reversal
is a fan prefix of the reflected opening sector \((L_2,R_0,L_1)\).  The
reversed marking is therefore the reflected walk with a prefix moved
before its first half-atom, and the prefix case just proved applies
verbatim; in particular the problematic state above is, up to sign and
reflection, a sector state of \eqref{eq:ledger-walk} rather than a
triangle-inequality estimate.  A marking with both a prefix and a
suffix is covered by the same two computations, which bound disjoint
portions of the walk: the early states by the prefix argument, the late
states by the reversal, and everything between by the displayed walk.
\end{proof}

\begin{proposition}[anchored calibration principle]
\label{prop:anchored-calibration}
Every standard polygonal escape path admitting an anchored support
marking has length at least \(C\).
\end{proposition}

\begin{proof}
Use the common source partition preceding
Lemma~\ref{lem:calibration}.  Equations
\eqref{eq:chunk-support}--\eqref{eq:chunk-balance} turn the support
integral into a finite list satisfying the support and zero-sum
hypotheses of that lemma.  Integrating the escape inequality on every
cell and adding bounds this finite total below by \(C\), as computed in
(4.2).  Lemma~\ref{lem:ledger} supplies its unit-disk suffix bound, and
Lemma~\ref{lem:calibration} bounds the same total above by the length of
the path.
\end{proof}

\begin{remark}[what is general, and what is golden]
The engine of this section is not tied to \(\beta=\pi/5\).  For any
triangle, its three outward normals have a positive linear dependence;
that dependence both eliminates translation from the escape criterion
and folds any positive source measure into balanced vector triples.
Normal-cone aggregation and Abel summation then apply verbatim.  The two
gnomon-specific achievements are the choice of source measure whose
ledger is sharp, and the rigidity argument forcing a shortest polygonal
counterexample into its safe temporal order.

The second of these is genuinely local, and quantifiably so.  The
rigidity argument is calibrated to \(C<129/100\), whereas in the
numerical continuation mentioned in Section~1 the best known escape
lengths along the same line--arc branch rise to about \(1.389\) at its
upper end.  Every numerical input of
Appendix~\ref{sec:rigidity} would therefore have to be recomputed at
another base angle, even though the calibration engine itself would be
unchanged.
\end{remark}

\section{The proof in three steps}\label{sec:proof-spine}

The finite candidate calculation is already summarized by
Proposition~\ref{prop:attainment}, and the entire analytic lower bound is
Proposition~\ref{prop:anchored-calibration}.  The remaining geometric
content is the following statement.
For points on an oriented path, write \(X\prec Y\) when \(X\) is
encountered before \(Y\).

\begin{proposition}[rigidity of a shortest counterexample]
\label{prop:rigidity}
Let \(N\ge1\), and let \(\eta\) be a standard polygonal escape path which
is length-minimal among polygonal escape paths with at most \(N\)
segments.  If \(\len\eta<C\), then \(\eta\) admits an anchored support
marking.
\end{proposition}

\begin{proof}
A path with no segment is a point and does not escape.  The one-segment
case is also impossible.  Translate one endpoint to the origin,
write the other as \(\ell e\) with \(|e|=1\), and choose \(n_0=u_\phi\)
perpendicular to \(e\).  Then \(h_\eta(n_0)=0\), while
\[
             e\cdot n_1=\pm s,\qquad e\cdot n_2=\mp s.
\]
Consequently
\(h_\eta(n_1)+h_\eta(n_2)=s\ell\), and (3.1) gives
\(\ell\ge2c\).  The certified bound \(C<129/100<2c\) rules this out.

For a convex path, apply an isometry so that its unique gap is an upper
support with outward normal \(u_{\pi/2}\), and reverse the path parameter
if necessary so that the complementary hull boundary is traversed in
normal-fan order.  Split the \(Z_0\)-atom between the two endpoint
contacts.  Every other support contact then occurs in
normal-fan order, so Lemma~\ref{lem:fan-ledger} gives an anchored
marking.  Let \(\eta\) be nonconvex.
Lemma~\ref{lem:lambda-gap} supplies an omitted hull edge \(FT\), an
opposite support vertex \(M\), and the distance \(h\) between their
parallel support lines, with
\[
                  F\prec M\prec T,\qquad
                  \len\eta\ge(1+\sqrt2)h,\qquad h<s .
\]

Align the gnomon's base line with the lower support and translate it
horizontally, first until the left oblique arm supports \(\eta\), and
then, in a separate configuration, until the right arm supports it.
In each configuration the opposite arm is only required to be crossed.
Lemma~\ref{lem:outer} shows that the two genuine supporting contacts
\(Z_1,Z_2\) lie on the open middle subpath \(\eta_{F,T}\).
Those metric estimates use the gap as a lower support.  Reflect the
entire configuration in the gap line.  The isometric image has the same
length and temporal ranks and is still an escape minimizer by
Lemma~\ref{lem:reflection-invariance}, but now \(FT\) is an upper support:
\(Z_0=(0,2\lambda c)\) supports the gap and the two reflected oblique
normals are the directions of \(Z_1,Z_2\) in
\eqref{eq:ledger-blocks}.  Around the convex hull, the temporal ranks of a standard
path are cyclically bitonic by Lemma~\ref{lem:bitonic}.  Cutting this
cyclic sequence at \(FT\) gives a decreasing--increasing--decreasing
sweep.  In normal order \(F<Z_1\le M\le Z_2<T\), while
\(t(F)<t(Z_1)\) and \(t(Z_2)<t(T)\).
These are exactly the hypotheses of Lemma~\ref{lem:sweep}: both anchors
lie on the central increasing phase, the earlier contacts can contribute
only one temporal prefix, and the later contacts only one temporal
suffix.  Lemma~\ref{lem:fan-ledger} identifies the resulting
prefix--middle--suffix order with an anchored support marking.
\end{proof}

\begin{proof}[Proof of Theorem~\ref{thm:main}]
Proposition~\ref{prop:attainment} gives
\(\mathcal E(G)\le C\).  If a shorter rectifiable escape path existed,
Proposition~\ref{prop:standard} would give an \(N\) and a standard
polygonal path of length less than \(C\), shortest among escape paths
with at most \(N\) segments.  Proposition~\ref{prop:rigidity} would make
its support marking anchored, and
Proposition~\ref{prop:anchored-calibration} would give the contradictory
lower bound \(\len\eta\ge C\).
\end{proof}

\begin{remark}[the equality mechanism]
An exact optimum requires the construction and the lower bound to meet,
not merely to agree numerically.  Here they meet on
\(\supp\nu\): on the candidate, the escape inequality (3.1) is tight
precisely at the source directions---the sharp window of
Proposition~\ref{prop:attainment}, its reflection, and the atom
direction---while the ledger walk rides the unit circle exactly on its
two arc blocks \eqref{eq:unit-ledger-blocks}.  The calibration
equations \eqref{eq:calibration-system} are precisely this meeting
condition.
\end{remark}

\appendix

\section{Technical proof of minimizer rigidity}\label{sec:rigidity}

This appendix proves the geometric and finite ingredients used in
Proposition~\ref{prop:rigidity}.  The geometric exclusion through
Lemma~\ref{lem:sweep} uses no data from the candidate beyond
\(C<1.29\) and \(C/(1+\sqrt2)<s\).  The final fan-to-ledger interface
uses only the coarse angular order \(a<b<0\) and
\(a>-\beta/2\); none of the candidate coordinates \((x,y,d)\) enters.

\subsection{Polygonal minimizers and their boundary order}

For a compact convex set \(H\), define
\[
 \Phi(H)=\min_{\phi\in\R}
 \{2c\,h_H(\phi)+h_H(\phi+\pi-\beta)+h_H(\phi+\pi+\beta)\}.
                                                                    \tag{A.1}
\]
The function being minimized is continuous and \(2\pi\)-periodic, so
the minimum is attained.
Thus \(H\) is an escape hull exactly when \(\Phi(H)\ge2sc\).
The balance identity \eqref{eq:balance} makes \(\Phi\) translation
invariant; monotonicity under inclusion and positive homogeneity follow
term by term from the support function.  If \(d_H\) denotes Hausdorff
distance, then
\(|h_H(u)-h_K(u)|\le d_H(H,K)\) for every unit vector \(u\), and hence
\[
 |\Phi(H)-\Phi(K)|\le(2c+2)d_H(H,K),
 \qquad \Phi(rH)=r\Phi(H)\quad(r\ge0).
\]

\begin{proposition}[standard polygonal reduction]\label{prop:standard}
If a rectifiable escape path of length less than \(C\) exists, then
there are an integer \(N\) and a shortest escape path \(\eta\), among
polygonal paths with at most \(N\) segments, such that
\(\len\eta<C\).  Moreover, \(\eta\) may be chosen \emph{standard}: it is
simple and its vertices are precisely the vertices of its convex hull,
each visited once.
\end{proposition}

\begin{proof}
Let \(\gamma\) be a rectifiable counterexample.  Reparameterize it on
\([0,1]\).  Starting from partitions whose chord sums approach
\(\len\gamma\), take successive unions with the preceding partitions
and with dyadic meshes.  The resulting partitions are refining, their
mesh tends to zero, and their chord sums still tend to
\(\len\gamma\): refinement cannot decrease a chord sum, while every
chord sum is at most \(\len\gamma\).  Their chordal interpolants
\(\gamma_n\) converge uniformly to \(\gamma\).  Indeed, if \(t\) lies
between consecutive partition points \(u,v\), then
\(\gamma_n(t)\) is a convex combination of \(\gamma(u),\gamma(v)\);
both are within
\(\omega_\gamma(\operatorname{mesh}\mathcal P_n)\) of \(\gamma(t)\).
Hence
\[
 \sup_{t\in[0,1]}|\gamma_n(t)-\gamma(t)|
 \le \omega_\gamma(\operatorname{mesh}\mathcal P_n)\longrightarrow0.
\]
Thus
\(\len\gamma_n\to\len\gamma\).
Uniform convergence gives Hausdorff convergence of the images.
Convexification is nonexpansive for Hausdorff distance---apply the same
convex coefficients to matched points---so the displayed estimate gives
\(\Phi(\conv\gamma_n)\to\Phi(\conv\gamma)\ge2sc\).
Put
\[
 r_n=\max\left\{1,\frac{2sc}{\Phi(\conv\gamma_n)}\right\}.
\]
For all sufficiently large \(n\), the denominator is positive and
\(r_n\to1\).  By homogeneity \(r_n\gamma_n\) escapes, while its length
tends to \(\len\gamma<C\).

Fix such a path \(P=(P_0,\ldots,P_N)\), translate \(P_0\) to the origin,
and put \(M=\len P<C\).  Padding shorter paths by repeated terminal
vertices, consider
\[
\mathcal K=\left\{(Q_0,\ldots,Q_N):
\begin{array}{l}
 Q_0=0,\quad \sum_{i=0}^{N-1}|Q_{i+1}-Q_i|\le M,\\
 \Phi\bigl(\conv\{Q_0,\ldots,Q_N\}\bigr)\ge2sc
\end{array}\right\}.
\]
Every \(Q_i\) lies in the closed ball of radius \(M\).  Both defining
maps are continuous.  More explicitly, if
\(\max_i|Q_i-Q_i'|\le\delta\), then every convex combination
\(\sum_i\alpha_iQ_i\) is within \(\delta\) of
\(\sum_i\alpha_iQ_i'\); hence the finite-hull map is one-Lipschitz in
the maximum vertex metric.  Thus \(\mathcal K\) is a nonempty compact
subset of that ball's \((N+1)\)-fold product.
Length therefore has a minimizer \(\eta\) on \(\mathcal K\).  It is also
shortest among all escaping paths with at most \(N\) segments: after
translation, every such path of length at most \(M\) lies in
\(\mathcal K\), while every other one is longer than \(M\ge\len\eta\).

Among the length minimizers choose one having the fewest vertices after
consecutive repetitions and terminal padding have been suppressed.
This second minimization removes all degeneracy at once.  If a vertex
occurs twice, delete either occurrence.  If a distinct vertex is not
extreme in the hull of the vertex set, then, because the set is finite,
it lies in the convex hull of the other vertices and may likewise be
deleted.  At an internal occurrence deletion replaces its two incident
segments by their chord; at an endpoint it deletes the single incident
segment.  In every case the hull is unchanged and the triangle
inequality does not increase length.  A strict decrease contradicts
length minimality, while equality contradicts the choice of the fewest
vertices.  If the initial vertex changes, translate the shortened path
back to the origin; this changes neither length nor \(\Phi\).
Consequently every vertex occurs once and is an extreme point of the
hull.

No three distinct hull vertices are collinear, since the middle one
would not be extreme.  Nor can an extreme vertex lie in the relative
interior of another segment.  Thus an intersection witnessing
nonsimplicity can be neither a repeated endpoint, an endpoint lying
inside a nonincident segment, nor a collinear overlap.  Any failure of
simplicity must therefore give two nonadjacent segments
\(DE\) and \(FG\), encountered in that temporal order, crossing in their
relative interiors.  Reverse the subpath from \(E\) to \(F\) and replace
\(DE,FG\) by \(DF,EG\).  The new path has the same endpoints, vertex set,
hull, and number of segments.  If \(X\) is the crossing point, neither
\(D,X,F\) nor \(E,X,G\) is collinear, so the two triangle inequalities
are strict:
\[
\begin{aligned}
 |DF|+|EG|
 &<(|DX|+|XF|)+(|EX|+|XG|)\\
 &=|DE|+|FG|.
\end{aligned}
\]
This contradicts length minimality.  Thus the chosen minimizer is simple,
visits every hull vertex exactly once, and has length at most \(M<C\);
it is standard.
\end{proof}

Write the hull vertices counterclockwise as
\(V_0,\ldots,V_{n-1}\), and let \(t(V_i)\) be their temporal ranks
on a standard path.

\begin{lemma}[endpoint peeling and cyclic bitonicity]\label{lem:bitonic}
After cyclically relabelling so that \(V_0\) is the initial endpoint,
write \(W_0,\ldots,W_{n-1}\) for the vertices in temporal order, so
\(W_0=V_0\).  For every \(j\ge1\), the remaining set
\[
                         \{W_j,\ldots,W_{n-1}\}
\]
is a consecutive sublist of the boundary list
\((V_1,\ldots,V_{n-1})\), and \(W_j\) is one of its two endpoints.
Consequently there is a \(k\) such that
\[
 t(V_0)<t(V_1)<\cdots<t(V_k)
 \quad\hbox{and}\quad
 t(V_k)>t(V_{k+1})>\cdots>t(V_{n-1})>t(V_0).                  \tag{A.2}
\]
\end{lemma}

\begin{proof}
We first isolate the only planar fact used.  The first segment \(XY\)
of a simple Hamiltonian path through a finite set in convex position is
an edge of that set's convex hull.  If not, each of the two open boundary
arcs from \(X\) to \(Y\) contains an unused vertex.  The tail beginning
at \(Y\) visits both arcs, so some first tail edge joins one open arc to
the other.  Its endpoints alternate around the boundary with \(X,Y\);
the two segments therefore cross in their relative interiors, contrary
to simplicity.  Only one geometric input is used here, namely that two
chords of a convex polygon with four distinct endpoints meet in their
relative interiors exactly when their endpoints alternate around the
boundary.  Granting that criterion, the deduction just given is a
statement about finite sequences, and in that form it is kernel-checked
as \code{first\_step\_is\_hull\_edge} in \code{ConvexOrder.lean}.

Apply this fact first to the full path.  It gives
\(W_1\in\{V_1,V_{n-1}\}\), so after \(V_0\) is deleted the current
vertex is an endpoint of the remaining boundary list.  Suppose
inductively that the remaining vertices form a consecutive list and
that its current vertex \(W_j\) is an endpoint.  Apply the same fact to
the simple Hamiltonian tail
\[
                         W_j,W_{j+1},\ldots,W_{n-1}.
\]
Within the convex hull of the remaining list, the two neighbours of
\(W_j\) are the next vertex inward and the opposite endpoint (the
closing chord).  Thus \(W_{j+1}\) is an endpoint after \(W_j\) is
deleted.  This proves the asserted end-deletion invariant by induction.

Let \(V_k=W_{n-1}\) be the last remaining vertex.  The vertices on the
\(V_1,\ldots,V_{k-1}\) side can only be deleted in that order, while
those on the \(V_{n-1},\ldots,V_{k+1}\) side can only be deleted in the
reverse boundary order.  Their two deletion streams may interleave, but
the inequalities within either stream are fixed.  They meet at
\(V_k\), whose rank is maximal, and this is exactly (A.2).
\end{proof}

Thus the temporal order is literally a deque order on the boundary
vertices.  In particular no separate support-contact theorem is needed:
as an oriented support line rolls around the polygon, its exposed vertex
has a cyclic bitonic temporal rank.  At a normal exposing an edge, its
two endpoints give the two one-sided branches.

\begin{remark}
The conclusion identifies the simple polygonal paths whose node set is the
vertex set of a strictly convex \(n\)-gon with the deque orders, and the
generation rule---choose an initial vertex, then one of two available
vertices at each later step---is exactly the one used to count them in
\cite[Lemma~2]{alexanderwetzelwichiramala2019}, where their number is
found to be \(n2^{n-2}\).  Lemma~\ref{lem:bitonic} is therefore a known
fact, recorded in the same source as
Theorem~\ref{thm:lambda}; the short proof is included only to keep the
treatment self-contained and to expose the end-deletion invariant, which
is what Lemmas~\ref{lem:outer-caps} and~\ref{lem:local-gap} actually use.
\end{remark}

\subsection{A quantitative gap}

The boundary edges omitted by a standard path are its \emph{gaps}; the
subpath joining the endpoints of a gap is its \emph{arch}.  A standard
path will be called \emph{convex} if it follows one of the two boundary
chains between its endpoints, equivalently if it has exactly one gap.
We need the
following local form of the Adhikari--Pitman surgery
\cite[Lemma~7]{adhikaripitman1989}.  We include the argument because our
selected gap comes from the \(\Lambda\)-property rather than from a
minimum-width direction.

\begin{lemma}[adjacent-gap outer cap]\label{lem:outer-caps}
Let \(\alpha\) be a standard polygonal path, let \(H=\conv\alpha\), and
put a gap \(C_1C_2\) on the lower support \(y=0\), with \(C_1\prec C_2\).
If \(C_2\) is not terminal, there is another gap \(L_1L_2\) such that
\[
                         C_1\prec L_1\prec C_2\prec L_2.
\]
The link entering \(C_2\) is \(L_1C_2\).  If the two gap lines are not
parallel, reflect in a vertical line if necessary so that \(C_1\) lies
to the left of \(C_2\), and let their intersection be \(E\).  Exactly one
of the following alternatives holds:
\[
\begin{array}{ll}
\text{\rm(R)}&
 E\text{ lies on the ray of }C_1C_2\text{ beyond }C_2,
 \quad\text{and the other line has order }L_1,L_2,E;\\
\text{\rm(L)}&
 E\text{ lies on the opposite base ray beyond }C_1,
 \quad\text{and the other line has order }E,L_1,L_2.
\end{array}
\]

In alternative \({\rm(R)}\), let \(A\) be the initial endpoint and let
\(P\) be the perpendicular projection of \(L_1\) to \(C_1C_2\).  The path
\[
                         \alpha_{A,L_1}\ast\seg{L_1}{E}
\]
has no more segments than \(\alpha\), and its convex hull contains \(H\).
If \(P\) lies on the closed base ray from \(E\) away from \(C_2\), the
same is true with \(E\) replaced by \(P\).
\end{lemma}

\begin{proof}
Use the end-deletion invariant of Lemma~\ref{lem:bitonic}.  Immediately
after \(C_1\) is selected, its boundary neighbour \(C_2\) is one endpoint
of the remaining boundary list.  Since \(C_2\) is not selected until
later, it remains that endpoint.  Let \(L_1\) be its temporal
predecessor.  Just before \(L_1\) is selected, the two endpoints of the
remaining list are therefore \(L_1\) and \(C_2\).  If \(L_2\) is the
next vertex inward from \(L_1\), deleting \(L_1\) leaves endpoints
\(L_2,C_2\), and the path next takes \(C_2\).  Hence \(L_1C_2\) is the
traversed cross-link and the boundary edge \(L_1L_2\) is omitted.
Because \(C_2\) is not terminal, \(L_2\ne C_2\) and is selected later.
This proves
\[
                         C_1\prec L_1\prec C_2\prec L_2.
\]
The corresponding boundary order, starting at \(C_1\) along the
unvisited side, is
\[
                         C_1,\ C_2,\ldots,L_2,\ L_1.
\]
The two nonparallel supporting lines bound two inward half-planes whose
intersection contains \(H\).  Their boundary rays from \(E\) enclose
that intersection, so the two exposed edges lie on the corresponding
rays beyond \(E\).  With \(C_1\) to the left of \(C_2\), the only two
possible orientations are
\[
                         C_1,C_2,E\quad\hbox{and}\quad L_1,L_2,E,
\]
or they have orders
\[
                         E,C_1,C_2\quad\hbox{and}\quad E,L_1,L_2.
\]
These are alternatives \({\rm(R)}\) and \({\rm(L)}\), respectively.

Assume \({\rm(R)}\).
Let \(\mathcal C\) be the boundary chain from \(C_1\) to \(L_1\) that
does not contain \(C_2\).  After selecting \(C_1\), endpoint peeling
proceeds along this chain until \(L_1\), when the link \(L_1C_2\)
switches to the other end of the remaining interval.  Thus every vertex
of \(\mathcal C\) has already occurred in the prefix through \(L_1\).
Moreover, \(EC_1C_2\) and \(EL_1L_2\) are the two supporting tangents
from \(E\) to \(H\).  Consequently the boundary of
\(\conv(H\cup\{E\})\) is
\[
              \seg{E}{C_1}\ \cup\ \mathcal C\ \cup\ \seg{L_1}{E}.
\]
Both \(C_1\) and \(L_1\), and every vertex of \(\mathcal C\), occur in the
prefix \(\alpha_{A,L_1}\).  Hence all three displayed boundary pieces
lie in \(\conv(\alpha_{A,L_1}\cup\{E\})\).  Taking their convex hull gives
\[
 H\subseteq\conv\bigl(\alpha_{A,L_1}\cup\{E\}\bigr).
\]
If \(P\) lies on the base ray beyond \(E\), then
\(E\in\seg{C_1}{P}\), so replacing \(E\) by \(P\) preserves this inclusion.
Each replacement discards the suffix after \(L_1\) and adds one segment;
because \(C_2\) is not terminal, its segment count is strictly smaller.
\end{proof}

\begin{lemma}[two-gap estimate]\label{lem:local-gap}
Let \(\alpha\) be a standard polygonal path which is length-minimal in
a class with the following property: replacing \(\alpha\) by a
polygonal path with no more segments and with convex hull containing
\(\conv\alpha\) preserves admissibility.  Suppose a gap \(C_1C_2\) lies
on one boundary of a supporting strip of width \(h\), the opposite
boundary meets the arch at \(M\), and
\[
                            C_1\prec M\prec C_2.
\]
If at least one of \(C_1,C_2\) is not an endpoint of \(\alpha\), then
\[
                         \len\alpha\ge(1+\sqrt2)h.              \tag{A.3}\label{eq:local-gap}
\]
\end{lemma}

\begin{proof}
The assertion is immediate for \(h=0\).  Time reversal and Euclidean
isometries are bijections on polygonal paths preserving length, segment
count, and convex-hull inclusion.  Transporting the admissible class by
either bijection therefore preserves both its replacement property and
the minimality of \(\alpha\).  Apply these bijections if necessary, and
relabel the gap endpoints in temporal order, so that \(C_2\) is not
terminal, \(C_1C_2\) lies on \(y=0\), and
\(\alpha\subset\{0\le y\le h\}\).  Denote the initial endpoint of
\(\alpha\) by \(A\), its terminal endpoint by \(A'\), write
\(\alpha_{X,Y}\) for a temporal subpath, and use \(\ast\) for
concatenation.  Apply Lemma~\ref{lem:outer-caps} and use its notation.

If \(L_1L_2\parallel C_1C_2\), it lies on the opposite support \(y=h\).
The three disjoint subpaths from \(C_1\) to \(L_1\), from \(L_1\) to
\(C_2\), and from \(C_2\) to \(L_2\) each cross the strip.  Projection
gives \(\len\alpha\ge3h\), which is stronger than
\eqref{eq:local-gap}.

\smallskip\noindent\emph{Excluding the left-facing alternative.}
Assume the gap lines are not parallel.  Alternative \({\rm(L)}\) of
Lemma~\ref{lem:outer-caps} cannot occur.  In that alternative the tangent
rays have orders \(E,C_1,C_2\) and \(E,L_1,L_2\).  The second ray leaves
the base into the upper half-plane, so
\[
                              (L_2)_y>(L_1)_y.                  \tag{A.4}\label{eq:left-facing-height}
\]
After \(C_1\) is removed in the endpoint-peeling description,
\(C_2\) is one endpoint of the remaining boundary interval.  Since it
is not selected until immediately after \(L_1\), every vertex selected
between \(C_1\) and \(L_1\) comes from the other endpoint and lies on the
opposite boundary chain; consecutive choices from that endpoint make
the temporal subpath \(\alpha_{C_1,L_1}\) follow this chain.  In
alternative \({\rm(L)}\), the boundary edge from \(L_2\) to \(L_1\)
points down and to the left.  By monotone turning of the
counterclockwise convex boundary, its direction angle and all subsequent
direction angles up to that of the rightward horizontal edge \(C_1C_2\)
lie in \([\pi,2\pi]\).  Every such edge therefore has nonpositive
vertical component.  Height is consequently nonincreasing from \(L_1\)
along that opposite chain back to \(C_1\).  Hence the whole subpath from
\(C_1\) to \(L_1\), and also the link \(L_1C_2\), has height at most
\((L_1)_y\).  The entire arch \(\alpha_{C_1,C_2}\) consequently has
height at most \((L_1)_y\), whereas \(L_2\in H\) has larger height by
\eqref{eq:left-facing-height}.  This contradicts the hypothesis that the
opposite support \(y=h\) meets that arch at \(M\).  Thus alternative
\({\rm(R)}\) holds.

\smallskip\noindent\emph{(i) Fixing the wedge.}
If \(E\) were not beyond \(P\),
then \(P\) would lie on the closed base ray from \(E\) away from
\(C_2\).  Lemma~\ref{lem:outer-caps} would make the first outer cap
admissible, while
\[
                        |L_1P|<|L_1C_2|
                              \le\len\alpha_{L_1,C_2}.
\]
Here \(P\) is the foot of the perpendicular from \(L_1\), so
\(|L_1P|\le|L_1C_2|\) always, and the inequality is strict because
\(P\ne C_2\): in this case \(P\) lies on that ray, so
\(P_x\ge E_x\), whereas alternative \({\rm(R)}\) puts \(E\) beyond
\(C_2\), giving \(P_x\ge E_x>(C_2)_x\).
It would therefore be strictly shorter than \(\alpha\), a contradiction.
Thus \(E\) lies beyond \(P\), and \(L_1E\) descends to the right as in
Figure~\ref{fig:two-gap}.

\begin{figure}[h]
\centering
\begin{tikzpicture}[scale=.82,line cap=round,line join=round]
  \begin{scope}
    \draw[gray!55] (-.2,0)--(6.7,0);
    \draw[gray!55] (-.2,2)--(6.7,2);
    \draw[very thick]
      (-.1,.22)--(.55,0)--(1.15,1.25)--(2.05,2)--(3,1.58)
      --(3.6,0)--(4.9,.722)--(5.7,.32);
    \draw[dashed,blue!55!black] (.55,0)--(3.6,0);
    \draw[dashed,blue!55!black] (3,1.58)--(4.9,.722);
    \draw[densely dotted,orange!85!black] (4.9,.722)--(6.5,0);
    \draw[densely dotted,orange!85!black] (3,1.58)--(3,0);
    \fill (.55,0) circle (1.4pt) node[below] {\scriptsize \(C_1\)};
    \fill (2.05,2) circle (1.4pt) node[above] {\scriptsize \(M\)};
    \fill (3.6,0) circle (1.4pt) node[below] {\scriptsize \(C_2\)};
    \fill (3,1.58) circle (1.4pt) node[above right] {\scriptsize \(L_1\)};
    \fill (4.9,.722) circle (1.4pt) node[above right] {\scriptsize \(L_2\)};
    \fill (3,0) circle (1.4pt) node[below] {\scriptsize \(P\)};
    \fill (6.5,0) circle (1.4pt) node[below] {\scriptsize \(E\)};
    \node at (3.1,-.9) {\small (a) the adjacent gaps and outer cap};
  \end{scope}
  \begin{scope}[xshift=7.7cm]
    \draw[gray!55] (-.25,0)--(5.45,0);
    \draw[dashed,blue!55!black] (.5,1.48)--(5.2,0);
    \draw[very thick] (.25,0)--(1.55,2)--(3.05,0)--(4,.38);
    \draw[densely dotted,orange!85!black] (1.55,2)--(1.55,-2);
    \draw[densely dotted,orange!85!black] (1.55,-2)--(5.2,0);
    \draw[densely dotted,orange!85!black] (3.05,0)--(4,-.38);
    \draw[densely dotted,orange!85!black] (4,-.38)--(5.2,0);
    \fill (.25,0) circle (1.4pt) node[below] {\scriptsize \(C_1\)};
    \fill (1.55,2) circle (1.4pt) node[above] {\scriptsize \(M\)};
    \fill (1.55,-2) circle (1.4pt) node[below] {\scriptsize \(M'\)};
    \fill (3.05,0) circle (1.4pt) node[below] {\scriptsize \(C_2\)};
    \fill (4,.38) circle (1.4pt) node[above left] {\scriptsize \(L_2\)};
    \fill (4,-.38) circle (1.4pt) node[below left] {\scriptsize \(L_2'\)};
    \fill (5.2,0) circle (1.4pt) node[below] {\scriptsize \(E\)};
    \draw (4.66,0) arc[start angle=180,end angle=162.4,radius=.46];
    \node at (4.52,.19) {\scriptsize \(\theta\)};
    \node at (2.45,-2.85) {\small (b) the reflection estimate};
  \end{scope}
\end{tikzpicture}
\caption{The shortened two-gap argument.  Blue dashed segments are the
two gaps; orange dotted segments are the outer cap or reflections.}
\label{fig:two-gap}
\end{figure}
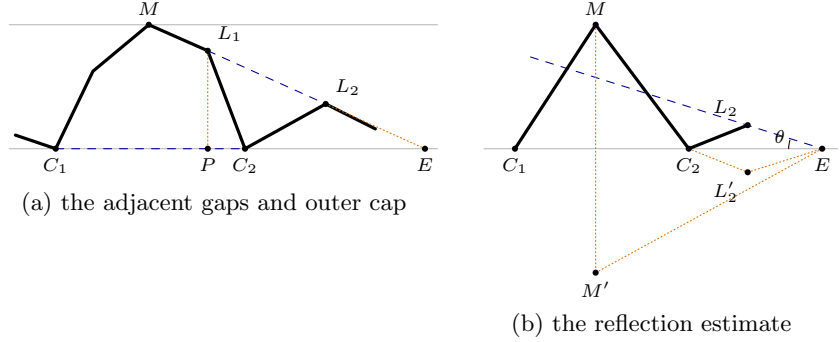

\smallskip\noindent\emph{(ii) The straight cap forces the angle.}
Minimality against the outer cap gives
\[
\begin{aligned}
 |L_1E|
 &\ge\len\alpha_{L_1,A'}\\
 &\ge\len\alpha_{L_1,C_2}+\len\alpha_{C_2,L_2}\\
 &\ge |L_1C_2|+|C_2L_2|.
\end{aligned}
\]
Put \(\theta=\angle C_2EL_2\), write
\[
 L_i=t_i(-\cos\theta,\sin\theta)\qquad(t_1>t_2>0),
\]
and reflect \(L_2\) in the base to \(L_2'\).  Since \(C_2\) lies on the
base, the preceding bound and the triangle inequality give
\[
 |L_1L_2'|
 \le |L_1C_2|+|C_2L_2'|
 =  |L_1C_2|+|C_2L_2|
 \le t_1.
\]
On the other hand,
\[
\begin{aligned}
 |L_1L_2'|^2
  &=(t_1-t_2)^2\cos^2\theta+(t_1+t_2)^2\sin^2\theta\\
  &=t_1^2+t_2^2-2t_1t_2\cos(2\theta).
\end{aligned}
\]
Thus \(0<t_2\le2t_1\cos(2\theta)\).  Since
\(0\le\theta<\pi/2\), this proves
\[
                              0\le\theta<\frac{\pi}{4}.          \tag{A.5}\label{eq:local-gap-angle}
\]

\smallskip\noindent\emph{(iii) Reflection.}
Put \(\alpha_0=\angle C_2EM\).  Since \(L_1L_2\) supports the hull,
\(M\) lies in the wedge between the two gap lines, and therefore
\(0\le\alpha_0\le\theta\).  Write \(d_0\) for the horizontal distance
from \(M\) to \(E\).  Since
\(h/d_0=\tan\alpha_0\le\tan\theta<1\), we have \(d_0>h\).
Reflect \(M,L_2\) in \(C_1C_2\), obtaining \(M',L_2'\).

Take \(E\) as origin and the base ray toward \(C_2\) as the negative
\(x\)-axis.  Then
\[
 M=(-d_0,h),\qquad M'=(-d_0,-h),\qquad
 L_2'=t_2(-\cos\theta,-\sin\theta).
\]
Writing \(\det\) for the oriented area form,
\[
\begin{aligned}
 \det(M',M)&=-2d_0h<0,\\
 \det(M',L_2')
   &=t_2(d_0\sin\theta-h\cos\theta)\ge0,
\end{aligned}
\]
where the last inequality is equivalent to
\(h/d_0\le\tan\theta\).  Thus \(M\) and \(L_2'\) lie on opposite sides of
the line \(EM'\), with equality allowed for \(L_2'\).  The continuous
broken line \(M\leadsto C_2\leadsto L_2'\) must meet \(EM'\), so its
length is at least
\[
 \operatorname{dist}(M,EM')
   =\frac{2hd_0}{\sqrt{d_0^2+h^2}}
   \ge\sqrt2\,h.                                                \tag{A.6}
\]
Explicitly, reflection and the triangle inequality give
\[
\len\alpha_{M,C_2}+\len\alpha_{C_2,L_2}
 \ge |MC_2|+|C_2L_2'|\ge\sqrt2\,h.
\]
The disjoint portion from \(C_1\) to \(M\) has length at least \(h\).
Adding proves
\eqref{eq:local-gap}: one full strip crossing, plus the
\(\sqrt2\,h\) diagonal cost forced by the second gap.
\end{proof}

We now choose the gap to which the estimate will be applied.  We use the
following result in exactly the form needed here.

\begin{theorem}[\(\Lambda\)-configuration theorem
{\cite[Theorem~5.1]{coultonmovshovich2006};
see also \cite[Theorem~1]{alexanderwetzelwichiramala2019}}]
\label{thm:lambda}
Let \(\gamma\) be a simple open polygonal arc of positive width.  There
are two parallel support lines, distinct points \(u,w\) of \(\gamma\) on
one line, and a point \(M\) of \(\gamma\) on the other, such that, after
naming the two outer points in temporal order,
\[
                              u\prec M\prec w .
\]
\end{theorem}

The cited statements assume only that \(\gamma\) is a simple open arc which
is not a line segment.  For a simple arc that is the same as positive
width: an arc of zero width lies in a line, and a simple arc contained in a
line is a segment.  Restricting to polygonal arcs, as we do, only weakens
the hypothesis further.

\begin{lemma}[selected \(\Lambda\)-gap]\label{lem:lambda-gap}
Let \(N\ge1\), and let \(\alpha\) be a nonconvex standard polygonal
escape path which is length-minimal among polygonal escape paths with at
most \(N\) segments.  If \(L=\len\alpha<C\), then there are distinct
parallel support lines at distance \(h>0\), a gap \(FT\) on one line,
and a hull vertex \(M\) on the other such that
\[
 F\prec M\prec T,\qquad
 L\ge(1+\sqrt2)h,\qquad
 h<\frac{C}{1+\sqrt2}<s<\frac{147}{250}.
\]
\end{lemma}

\begin{proof}
Standardness makes \(\alpha\) simple, open and polygonal.  It has
positive width: otherwise its hull would be a segment, whose only
extreme points are its endpoints, so a standard path would itself be
one segment and hence convex.  Apply
Theorem~\ref{thm:lambda}, and denote its two support lines and three
points by \(\ell,\ell'\) and \(u,M,w\).

We first identify \(u,w\) with the endpoints of a hull edge.  Let
\(\ell\) be the line containing \(u\) and \(w\).  If some segment
of \(\alpha\) were contained in \(\ell\), its endpoints would be hull
vertices in the exposed face \(\ell\cap\conv\alpha\).  A polygonal face
has only its two endpoint vertices here, because the standardization
removed collinear nonextreme vertices.  The segment would therefore be
the whole exposed edge, and it would be the only portion of \(\alpha\)
on \(\ell\).  The two points \(u,w\) would lie on this single temporal
segment, forcing every point of the intervening subarc---in particular
\(M\)---to lie on \(\ell\), a contradiction.

Thus no segment of \(\alpha\) lies in \(\ell\), so
\(\alpha\cap\ell\) consists only of hull vertices.  The exposed face
cannot be a single vertex because \(u\ne w\); it is therefore an edge
with endpoints \(F,T\), and \(\{u,w\}=\{F,T\}\).  Relabelling so that
\(F\prec T\),
\[
                              F\prec M\prec T.                  \tag{A.7}\label{eq:lambda-order}
\]
Since no segment of \(\alpha\) lies in \(\ell\), the edge \(FT\) is not
traversed: it is a gap, \(M\) lies on its arch, and the distance \(h\)
between the two support lines is exactly that arch's altitude.
We may take \(M\) to be a hull vertex: if it lies inside a path segment,
linearity puts both endpoints on the same support line.  Since that line
is distinct from the one through \(F,T\), the segment occurs in the open
subarc \(\alpha_{F,T}\); replacing \(M\) by either endpoint preserves
\eqref{eq:lambda-order} and \(h\).

Next, at least one of \(F,T\) is not an endpoint of \(\alpha\).  Indeed,
if both were, then the two endpoints of \(\alpha\) would be adjacent on
the hull; in the notation of Lemma~\ref{lem:bitonic} this forces
\(k=1\) or \(k=n-1\), so \(\alpha\) traverses the hull vertices
monotonically in one sense or the other and is convex, contrary to
assumption.  This is exactly the hypothesis under which the surgery in
Lemma~\ref{lem:local-gap} operates, at whichever endpoint of the gap is
not terminal.

Let \(\mathcal A_N\) be the class of polygonal paths \(\zeta\) with at
most \(N\) segments and
\(\Phi(\conv\zeta)\ge2sc\).  By the escape criterion this is exactly the
relevant admissible class, and \(\alpha\) is length-minimal in it.  If a
replacement has no more segments and its hull contains
\(\conv\alpha\), monotonicity of \(\Phi\) keeps that replacement in
\(\mathcal A_N\).  Thus every hypothesis of
Lemma~\ref{lem:local-gap} is explicit, and it gives
\[
                              L\ge(1+\sqrt2)h.                  \tag{A.8}\label{eq:lambda-bound}
\]

Since \(L<C\), \eqref{eq:lambda-bound} and the certified bounds give
\[
        h<\frac{C}{1+\sqrt2}<s<\frac{147}{250}.                \tag{A.9}\label{eq:height}
\]
\end{proof}

\subsection{A rational support estimate for all exceptional orders}

Put
\[
                    h_{\max}=\frac{147}{250},\qquad
                    \kappa=\frac{129}{100}.
\]

Put the lower support of the \Lambdaarch\ on \(y=0\), with the path in
\(y\ge0\), and orient it so that \(F\) lies to the left of \(T\).
Write \(k=\cot\beta=c/s\).  We first support the path by the left
oblique arm.  Choose the horizontal position of the left base vertex by
the exact condition
\[
                  \min_{P\in\alpha}\{x_P-k y_P\}=0,
\]
and let \(B\) be a minimizer, which exists by compactness.  In the
resulting coordinates the path lies weakly to the right of the left arm:
\[
                              x\ge k y,
\]
with equality at \(B\).  Because every point of the path has
\(0\le y\le h<s\), the right-arm clearance at \(B\) is positive:
\[
 r(B):=2c-k y_B-x_B=2c-2k y_B>0.
\]
The path must meet the right arm \(r=0\).  If some point had \(r<0\),
connectedness and \(r(B)>0\) would already give a meeting point.
Otherwise, failure to meet would mean \(r>0\) everywhere; compactness
would give \(\varepsilon:=\min_\alpha r>0\).  Choose
\(\Delta y>0\) and \(\Delta x>k\Delta y\) so small that
\(\Delta x+k\Delta y<\varepsilon\).  Translating the path by
\((\Delta x,\Delta y)\) makes
\[
\begin{aligned}
 y+\Delta y&>0,\\
 x+\Delta x-k(y+\Delta y)&>0,\\
 2c-k(y+\Delta y)-(x+\Delta x)&>0.
\end{aligned}
\]
Thus the translated path would lie in \(\operatorname{int}G\),
contrary to escape.  Choose any meeting point \(D\).  Notice the
deliberate asymmetry: \(B\) is a genuine support contact, whereas \(D\)
need only be a crossing of the opposite arm.  This is all the length
estimate below uses.

In coordinates with the left base vertex at the origin, the five
recorded points have the form
\[
\begin{aligned}
F&=(x_F,0),& T&=(x_T,0),\\
B&=(k y_B,y_B),&
D&=(2c-k y_D,y_D),\\
M&=(x_M,h).
\end{aligned}                                                    \tag{A.10}\label{eq:five-points}
\]
Every such shorter minimizer satisfies
\[
 x_F\le x_T,\qquad 0\le y_B,y_D\le h\le h_{\max}.
                                                              \tag{A.11}\label{eq:tetral-feasible}
\]
No bound on \(x_M\) is needed: it cancels from every estimate below, and
\(x_F,x_T\) enter only through their difference.

The same construction can be made with the right arm supporting instead.
Reflect in the vertical axis, reverse the path parameter, and relabel
\(F,T\); this preserves \(x_F\le x_T\) and
\(F\prec M\prec T\), and reduces that second placement to the one just
described.  Thus it is enough to prove the following estimate once, for
a pair \(B,D\) on the two arms of one placement, without assuming that
both are support contacts.

\medskip
\begin{lemma}[outer-order exclusion]\label{lem:outer}
If \(F\prec M\prec T\) and at least one of \(B,D\) does not lie in the
open temporal interval \((F,T)\), then every polygonal chain through the
five points in their temporal order has length greater than \(\kappa\).
When a label coincides with \(F\) or \(T\), order that label on the
adjacent outer side.
\end{lemma}

\begin{proof}
We shall use only
\[
 c>\frac{809}{1000},\qquad
 \frac{11}{8}<k<\frac{1377}{1000}<\frac32,\qquad
 0\le h\le\frac{147}{250}.                              \tag{A.12}\label{eq:coarse-outer}
\]
There are twenty orders of \(F,B,M,D,T\) compatible with
\(F\prec M\prec T\).  Six have both \(B,D\) between \(F,T\); the other
fourteen are exceptional.  Reflection in the vertical axis followed by
reversal of time exchanges \(F\leftrightarrow T\) and
\(B\leftrightarrow D\), while fixing \(M\).  Thus the fourteen raw
orders form eight classes.  Explicitly, the six two-element orbits and
the two fixed orders are
\[
\begin{gathered}
 BFMDT\leftrightarrow FBMTD,\qquad
 BFDMT\leftrightarrow FMBTD,\qquad
 FMDTB\leftrightarrow DFBMT,\\
 FMTBD\leftrightarrow BDFMT,\qquad
 FMTDB\leftrightarrow DBFMT,\qquad
 FDMTB\leftrightarrow DFMBT,\\
                         BFMTD,\qquad DFMTB .
\end{gathered}                                                   \tag{A.13a}\label{eq:outer-order-orbits}
\]
This also covers coincident labels: refine
any weak temporal order compatibly, using the outer-side convention in
the statement at \(F,T\).  Each equality then contributes a zero chord,
so the corresponding supported sum and all inequalities below are
unchanged.

We give one representative of each class.  Put
\[
\begin{gathered}
 e_x=(1,0),\quad e_y=(0,1),\quad
 A_\pm=\left(\frac45,\pm\frac35\right),\\
 U=\left(\frac25,\frac9{10}\right),\quad
 X=\left(\frac25,0\right),\quad
 C_0=\left(\frac45,-\frac25\right).
\end{gathered}
\]
\par\medskip\noindent
For the four successive chords, in temporal order, use the following
vectors:
\[
\renewcommand{\arraystretch}{1.16}
\begin{array}{c|c}
\text{order}&(q_0,q_1,q_2,q_3)\\ \hline
BFMDT&(A_-,A_+,C_0,-e_y)\\
BFDMT&(A_-,A_+,0,-e_y)\\
FMDTB&(U,X,-U,-A_-)\\
FDMTB&(U,-X,-U,-A_-)\\
FMTBD&(e_y,-e_y,e_y,e_x)\\
FMTDB&(e_y,-e_y,e_y,-e_x)\\
BFMTD&(A_-,A_+,A_-,A_+)\\
DFMTB&(-A_+,e_y,-e_y,-A_-)
\end{array}                                                     \tag{A.13}\label{eq:rational-supports}
\]
\par\medskip\noindent
Every vector has norm at most one.  The only nonunit checks are
\[
 |U|^2=\frac{97}{100},\qquad |X|^2=\frac4{25},\qquad
 |C_0|^2=\frac45.
\]
The directions are deliberately coarse: their only job is to clear the
barrier \(\kappa\) with an exactly computable rational margin, and the
simple components keep the expansion checkable by hand.
Therefore the elementary inequalities
\[
 |P_{j+1}-P_j|\ge q_j\cdot(P_{j+1}-P_j)
\]
give lower bounds for the chain.  Substitution of
\eqref{eq:five-points}, followed only by
\(x_F\le x_T\) and \(y_B,y_D\le h\), gives
\[
\renewcommand{\arraystretch}{1.18}
\begin{array}{c|c}
\text{orders}&\text{lower bound}\\ \hline
BFMDT,\ BFDMT&
 \displaystyle \frac{8c}{5}+\frac{11-8k}{5}\,h\\[2pt]
FMDTB,\ FDMTB,\ BFMTD&
 \displaystyle \frac45\bigl(2c+(3-2k)h\bigr)\\[2pt]
FMTBD,\ FMTDB&
 \displaystyle 2c+(3-2k)h\\[2pt]
DFMTB&
 \displaystyle \frac85\bigl(c+(2-k)h\bigr).
\end{array}                                                     \tag{A.14}\label{eq:rational-support-bounds}
\]
Every use of \(y_B,y_D\le h\) here replaces a variable having a
nonpositive coefficient; the signs follow from \(k>11/8\).
For example, the supported sum in either of the first two orders is
\[
 \frac{8c}{5}+h+
 \left(\frac35-\frac{4k}{5}\right)(y_B+y_D);
\]
the first line of \eqref{eq:rational-support-bounds} follows because its
last coefficient is negative.  The other three lines are the same direct
expansion; any omitted \(x_T-x_F\) term has a nonnegative coefficient.

By \eqref{eq:coarse-outer}, the last three lines of
\eqref{eq:rational-support-bounds} are greater than or equal to
\(8c/5>1.2944>\kappa\).  In the first line \(11-8k<0\), so the height cap
and the other bounds in \eqref{eq:coarse-outer} give
\[
\begin{aligned}
\frac{8c}{5}+\frac{11-8k}{5}h
&>
\frac85\frac{809}{1000}
+\frac{11-8(1377/1000)}5\frac{147}{250}\\
&=\frac{100978}{78125}
 =1.2925184>\kappa .
\end{aligned}
\]
This proves all eight representatives and hence all fourteen exceptional
orders.  Finally, splitting the competitor at the five contact parameters
and applying the triangle inequality shows that its length dominates the
corresponding four-chord chain.
\end{proof}

Because \(L<C<\kappa\), apply Lemma~\ref{lem:outer} first with the
left arm supporting.  Its genuine support contact \(B\) must lie on the
open middle subarc \(\alpha_{F,T}\); call it \(Z_1\).  Apply the reflected
construction with the right arm supporting.  Its genuine support
contact likewise lies on \(\alpha_{F,T}\); call it \(Z_2\).  The auxiliary
crossing point in either application is used only to invoke the
five-point estimate and is then discarded.
Choose \(Z_1,Z_2\) as endpoint vertices of their exposed faces.
Lemma~\ref{lem:outer} rules out \(F,T\) for either choice, so
\(t(M),t(Z_1),t(Z_2)\) below are genuine vertex ranks.

\subsection{The anchored support sweep}

Two orders coexist on a standard path: the normal-fan order of its hull
boundary and the temporal order of its traversal.  For a path that
follows the boundary they agree; for a path that jumps across its hull
they can disagree violently.  The remaining work shows that the two
anchors leave room for exactly the two migrations that
Lemma~\ref{lem:ledger} tolerates.

The estimates above used a \emph{metric frame}: the gap \(FT\) was the
lower support \(y=0\), because that makes the gnomon placement and the
height \(h\) transparent.  The calibration uses the reflected
\emph{ledger frame}.  Reflect the entire path in the line \(FT\), retain
the same labels and temporal parameters, and prove the lower bound for
this isometric image.  By Lemma~\ref{lem:reflection-invariance} it is
again an escape minimizer.  Now \(FT\) is an upper support, the hull
lies in \(y\le0\), and
\[
 Z_0=(0,2\lambda c)
\]
is genuinely supported at \(F,T\).  The old left- and right-arm normals
\((-s,c)\) and \((s,c)\) become
\[
 \frac1\lambda Z_1=(-s,-c),\qquad
 \frac1\lambda Z_2=(s,-c),
\]
exactly the two atomic directions in \eqref{eq:ledger-blocks}.  The
contacts \(Z_1,Z_2\) are the \emph{anchors}: their certified temporal
position inside \((F,T)\) is all the rigidity the sweep below uses.
Reflection changes neither length, temporal order, the three-phase
property, nor the conclusions of Lemma~\ref{lem:outer}.

Orient the normal fan so that the complementary hull boundary runs from
\(F\), through the opposite support \(M\), to \(T\).  The omitted upper
gap edge returns from \(T\) to \(F\), while \(t(F)<t(T)\).  Cutting the
cyclic bitonic sequence of Lemma~\ref{lem:bitonic} along that edge
produces three phases on the complementary chain.  Indeed, the cyclic sequence
has one increasing branch from its minimum to its maximum and one
decreasing branch back.  The directed edge \(T\to F\) decreases temporal
rank, so it lies on the latter branch; deleting it leaves a suffix of
the decreasing branch, the full increasing branch, and a prefix of the
decreasing branch:
\[
             F\ \searrow\ t_{\min}\ \nearrow\ t_{\max}\
             \searrow\ T.                                     \tag{A.15}\label{eq:three-phase}
\]
In this subsection \(X\le Y\) denotes weak order along that oriented
complementary boundary; temporal comparisons are always written with
\(t\).  The hull lies to the right of the supporting left arm and to the
left of the supporting right arm.  In the ledger frame their outward
unit normals, together with that of the opposite lower line, are
\[
 n_L=(-s,-c)=u_{3\pi/2-\beta},\qquad
 n_D=(0,-1)=u_{3\pi/2},\qquad
 n_R=(s,-c)=u_{3\pi/2+\beta}.
\]
Along the complementary boundary from \(F\) to \(T\), the outward normal
rotates counterclockwise; hence the exposed faces containing \(Z_1,M,Z_2\)
occur in that order.  The outer-order exclusion makes the first and last
contacts strictly interior.  Consequently
\[
                         F<Z_1\le M\le Z_2<T.                  \tag{A.16}\label{eq:normal-anchors}
\]
The middle inequalities are weak because several directions can expose
the same polygon vertex.
The outer-order exclusion gives
\[
                         t(F)<t(Z_1),\qquad t(Z_2)<t(T).        \tag{A.17}\label{eq:time-anchors}
\]

\medskip
\begin{lemma}[anchored sweep]\label{lem:sweep}
Under \eqref{eq:normal-anchors}--\eqref{eq:time-anchors}, both
\(Z_1,Z_2\) lie on the central increasing phase of
\eqref{eq:three-phase}.  Among contacts \(X\le Z_1\), those with
\(t(X)\le t(F)\) form one normal-fan prefix.  Among contacts
\(Z_2\le X\), those with \(t(T)\le t(X)\) form one normal-fan suffix.
After deleting these two sets, all remaining contacts occur in weak
normal-fan order.
\end{lemma}

\begin{proof}
Write \(L,H\) for the two phase-change contacts in
\eqref{eq:three-phase}.  On \([F,L]\) the temporal rank \(t\) is weakly
decreasing.  Hence \(Z_1\le L\) would imply
\(t(Z_1)\le t(F)\), contrary to \eqref{eq:time-anchors}; therefore
\(L<Z_1\).  Similarly, \(H\le Z_2\) would imply
\(t(T)\le t(Z_2)\), so \(Z_2<H\).  Since
\[
                         L<Z_1\le M\le Z_2<H,
\]
both anchors and every contact between them lie on the central phase,
where \(t\) is weakly increasing.

It remains to identify the tails, rather than merely draw them.  Before
\(Z_1\), the predicate \(t(X)\le t(F)\) holds throughout the initial
decreasing phase.  On the central increasing phase it changes from true
to false at most once.  It cannot become true on the final decreasing
phase, because that phase ends at \(t(T)>t(F)\), so every one of its
ranks is at least \(t(T)\).  Hence the contacts with
\(t(X)\le t(F)\) form one normal-fan prefix, and every contact remaining
before \(Z_1\) is encountered in weak normal order.

Dually, \(t(T)\le t(X)\) changes from false to true at most once on the
central increasing phase and remains true throughout the final
decreasing phase, whose ranks decrease to \(t(T)\).  Those contacts form
one normal-fan suffix.  After deleting the prefix and suffix, all
remaining contacts lie on the increasing phase and therefore occur in
weak normal order.  This proves all three assertions.  The same
order-theoretic deduction, including the possible equalities at an
exposed edge, is formalized as
\code{ThreePhaseSweep.anchors\_force\_ledger\_shape}.
\end{proof}

One anchor would not suffice.  On a decreasing phase, normal order and
temporal order run oppositely, so vector pieces there enter the ledger
in reversed fan order; it is the pair of anchors, one on each side of
the bottom normal, that confines every out-of-order contact to the two
light tails.

Coincident vertex contacts introduce no additional case.  A direction
exposing an edge may be assigned either endpoint; if a chunk boundary
falls there, split the chunk.  If an atom has coincident contacts, either
temporal assignment gives the same support value.  In particular, split
the upper-gap \(Z_0\)-atom equally between \(F\) and \(T\), which replaces
one vector of the family in Lemma~\ref{lem:calibration} by two parallel
vectors.  Since the competitor is a polygon, all of this involves only
finitely many choices.

\begin{lemma}[folded fan-to-ledger interface]\label{lem:fan-ledger}
In the ledger frame, cut the normal fan at the upper-gap direction
\(\pi/2\) and split the two
continuous source intervals whenever one of their three support contacts
changes.  If the resulting contact sweep is in normal-fan order, or has
the prefix--middle--suffix form of Lemma~\ref{lem:sweep}, then the three
vector pieces of these chunks, sorted by temporal contact, have exactly an
order allowed by Lemma~\ref{lem:ledger}.  In particular, every suffix state
in Lemma~\ref{lem:calibration} has norm at most \(1\).
\end{lemma}

\begin{proof}
We first remove any ambiguity about the word ``exactly.''  Unwrap one turn
of the angular support of \(\mu\), beginning with its atom in direction
\(\pi/2\).  The ten vector blocks and their angular supports are
\[
\begin{array}{c|c@{\qquad}c|c}
\text{block}&\text{angular support}&\text{block}&\text{angular support}\\ \hline
\frac12Z_0&\{\pi/2\}&
R_1&{[\pi-\beta+a,\pi-\beta+b]}\\
L_0&{[\pi-b,\pi-a]}&
R_2&{[\pi+\beta+a,\pi+\beta+b]}\\
Z_1&\{3\pi/2-\beta\}&
Z_2&\{3\pi/2+\beta\}\\
L_1&{[2\pi-\beta-b,2\pi-\beta-a]}&
R_0&{[2\pi+a,2\pi+b]}\\
L_2&{[2\pi+\beta-b,2\pi+\beta-a]}&
\frac12Z_0&\{5\pi/2\}.
\end{array}                                                       \tag{A.18}\label{eq:folded-fan-order}
\]
This table follows directly from \eqref{eq:fold}: the unshifted right and
left source intervals give \(R_0,L_0\), while rotation by
\(\pi-\beta\) gives \(R_1,L_1\), and rotation by \(\pi+\beta\) gives
\(R_2,L_2\).  The three images of the atom are \(Z_0,Z_1,Z_2\).

The intervals in \eqref{eq:folded-fan-order} are strictly ordered as
displayed.  Lemma~\ref{lem:length} gives
\(-\beta/2<a<b<0\), which proves each successive comparison in the
table.  For example, the only comparisons
between two continuous blocks reduce to
\[
 2b<\beta,\qquad -2a<\beta,
\]
and those involving \(Z_1,Z_2\) reduce to
\[
 b<\frac{\pi}{2}-2\beta=\frac{\pi}{10},
 \qquad
 a>\beta-\frac{\pi}{2}=-\frac{3\pi}{10}.
\]
Thus the normal-fan block order is precisely
\[
 \frac12Z_0,\ R_1,\ L_0,\ R_2,\ Z_1,\ Z_2,\
 L_1,\ R_0,\ L_2,\ \frac12Z_0,                                 \tag{A.19}
\]
which is \eqref{eq:ledger-walk}.  The three strict parameter inequalities
used here are also kernel-checked in the theorem
\[
                    \code{certified\_folded\_fan\_separation}.
\]

For a normal-ordered contact sweep, splitting at contact changes merely
subdivides the blocks in this list, so Lemma~\ref{lem:ledger} applies
directly.  Now suppose Lemma~\ref{lem:sweep} applies.  By
\eqref{eq:folded-fan-order}, the entire normal interval before \(Z_1\)
consists of the side sector
\(\mathcal A=(R_1,L_0,R_2)\).  Let \(\mathcal P\) be the pieces in that
sector whose contact rank is at most \(t(F)\).  Lemma~\ref{lem:sweep}
says that \(\mathcal P\) is a prefix of \(\mathcal A\), and that the
relative order of \(\mathcal A\setminus\mathcal P\) is unchanged.  Split
the last source chunk if the threshold passes through it.  Dually, in
\(\mathcal B=(L_1,R_0,L_2)\), the pieces whose rank is at least \(t(T)\)
form a suffix \(\mathcal S\).  Equal-rank pieces all support at the same
polygon vertex and may be put on either side of the corresponding
half-atom.

After sorting by temporal contact, the complete vector list is therefore
\[
 \mathcal P,\ \frac12Z_0,\
 (\mathcal A\setminus\mathcal P),\ Z_1,\ Z_2,\
 (\mathcal B\setminus\mathcal S),\ \frac12Z_0,\ \mathcal S.     \tag{A.20}
\]
Every displayed portion retains its normal-fan order.  Relative to
\eqref{eq:ledger-walk}, (A.20) does only two things: it moves the prefix
\(\mathcal P\) across the first half of \(Z_0\), and the suffix
\(\mathcal S\) across the last half.  These are precisely the two
migrations allowed in Lemma~\ref{lem:ledger}; no appeal to the drawing,
and no further permutation, is hidden here.  Hence
\(|R_i|\le1\) for every finite suffix state.
\end{proof}

\section{Exact data and verification}\label{sec:exact-data}

\subsection{The quartic parameter}

The geometry above used only the three equations
\eqref{eq:calibration-system}.  For an exact, one-variable specification,
put
\[
\begin{aligned}
 \mathcal Q(t)={}&(7+30c)(t^4+1)+(94+300c)(t^3+t)+(134+428c)t^2\\
                 &-4sc\,(t^2-1)(t^2-4t+1).
\end{aligned}                                                    \tag{B.1}\label{eq:calibration-quartic}
\]
This quartic has a unique zero
\[
                      p\in\left[-\frac18,-\frac19\right].        \tag{B.2}\label{eq:p-definition}
\]
Here is the exact uniqueness argument, including the polynomial used by
the certificate.  Put
\[
\begin{aligned}
\mathcal P(t)={}&4sc(t^5+3t^4+94t^3+138t^2+93t+7)\\
 &-2c(15t^5+149t^4+218t^3+150t^2+11t+1)\\
 &+16s(t^4+7t^3+10t^2+7t+1)\\
 &-7t^5-97t^4-122t^3-94t^2-19t+3.
\end{aligned}
\]
Exact polynomial reduction using
\(4c^2=2c+1\) and \(4s^2=3-2c\) gives
\[
                         c\mathcal P(t)=(s-ct)\mathcal Q(t).      \tag{B.3}\label{eq:quintic-deflation}
\]
Let \(I=[-1/8,-1/9]\) and, as exact terminating decimals, put
\[
 r_-=-0.1200344643529,\qquad r_+=-0.1200344643527.
\]
Direct rational comparison gives
\(-1/8<r_-<r_+<-1/9\).
The exact inequalities used are
\[
 \mathcal P(r_-)<0<\mathcal P(r_+),\qquad
 \mathcal P'(t)>0\quad(t\in I),\qquad
 s-ct>0\quad(t\in I).                                           \tag{B.4}\label{eq:quintic-isolation}
\]
Only the first pair requires evaluation at the tight algebraic
enclosures.  The derivative has a large elementary margin.  For
\(r=-t\in[1/9,1/8]\), direct expansion gives
\[
\begin{aligned}
\mathcal P'(t)&=A(r)sc+B(r)c+C(r)s+D(r),\\
A(r)&=20r^4-48r^3+1128r^2-1104r+372,\\
B(r)&=-150r^4+1192r^3-1308r^2+600r-22,\\
C(r)&=-64r^3+336r^2-320r+112,\\
D(r)&=-35r^4+388r^3-366r^2+188r-19.
\end{aligned}
\]
Termwise use of \(1/9\le r\le1/8\) gives
\[
 A(r)>230,\qquad B(r)>20,\qquad C(r)>70,\qquad D(r)>-4.
\]
Together with \(sc>47/100\), \(c>4/5\), and \(s>29/50\), this yields
\[
 \mathcal P'(t)>
 230\frac{47}{100}+20\frac45+70\frac{29}{50}-4
 =\frac{1607}{10}>0.
\]
The last inequality in \eqref{eq:quintic-isolation} is immediate from
\(t<0\) and \(s,c>0\).  The Lean development checks these elementary
bounds as well as the exact sign bracket.
Thus the intermediate value theorem gives a zero of \(\mathcal P\) in
\([r_-,r_+]\), while the derivative inequality makes it the only zero
in \(I\).  Since \(c>0\), \eqref{eq:quintic-deflation} and the last
inequality in \eqref{eq:quintic-isolation} show that \(\mathcal P\) and
\(\mathcal Q\) have exactly the same zeros in \(I\).  This proves
\eqref{eq:p-definition}.  The discarded linear factor has the spurious
root \(t=s/c=\tan\beta>0\), outside the geometric branch.
Define
\[
\begin{aligned}
\mathcal N(t)={}&4ct^2s+4ct^2+24cts+4cs-4c
       +2t^2s-t^2+4ts+2s+1,\\
\mathcal D(t)={}&12ct^2s+12ct^2+24ct+20cs+12c
       +6t^2s+3t^2+10t+2s+3,
\end{aligned}
\]
and then
\[
\begin{aligned}
q&=\frac{\mathcal N(p)}{\mathcal D(p)},&
a&=2\arctan p,& b&=2\arctan q,\\
\lambda&=\frac1s\left(
 \sin a+1+\frac{\sin(a+\beta)-\sin(b+\beta)}{2c}\right).
\end{aligned}                                                    \tag{B.5}\label{eq:parameter-definition}
\]
The denominator is safely positive without a delicate estimate.  Indeed,
for \(-0.121<p<-0.119\), all four \(p^2\)-terms in
\(\mathcal D(p)\) are nonnegative, while \(0<c<1\), \(c>4/5\), \(s>0\)
give
\[
 \mathcal D(p)>12c+3+24cp+10p
              >\frac{48}{5}+3-\frac{34\cdot121}{1000}
              =\frac{4243}{500}>0.                              \tag{B.6}\label{eq:denominator-positive}
\]

Here is a human-checkable reconstruction of all three calibration
equations.  Let \(F_1,F_2,F_3\) denote the left sides of
\eqref{eq:calibration-system} after
\(a=2\arctan p,\ b=2\arctan q\), and put
\[
 \Delta=c(1+p^2)(1+q^2),\qquad
 \mathcal A=\Delta F_1,\qquad
 \mathcal B=\Delta F_2,\qquad
 \mathcal G=-\Delta F_3 .
\]
The half-angle identities give
\[
 \mathcal G(p,q)=3cp^2q-cpq^2-cp+3cq-p^2s+q^2s.
\]
The combination \(c\mathcal A-s\mathcal B\) cancels \(\lambda\).
Its remaining numerator, and the only auxiliary coefficient needed
below, factor compactly as
\[
\begin{aligned}
\mathcal H(p,q)={}&
 c^2\!\left((p^2+1)(q^2-q+1)+3p(q^2+1)\right)\\
&+cs(p^2+1)(q^2-1)+s^2(p-q)(pq-1),\\
\mathcal K(p)={}&c^2(p^2+3p+1)+cs(p^2+1)-ps^2 .
\end{aligned}                                                    \tag{B.7}
\]
Direct expansion, using only \(4c^2=2c+1\) and
\(4s^2=3-2c\) in the first two lines, gives the three identities
\[
\begin{aligned}
 \mathcal D(p)^2\,
 \mathcal G\!\left(p,\frac{\mathcal N(p)}{\mathcal D(p)}\right)
       &=2\mathcal P(p),\\
 (cp-s)\mathcal H(p,q)
   &=\left(\frac{\mathcal D(p)q-\mathcal N(p)}8\right)(p^2+1)
      -\mathcal K(p)\mathcal G(p,q),\\
 c\mathcal A-s\mathcal B&=-\mathcal H .
\end{aligned}                                                    \tag{B.8}\label{eq:reconstruction-chain}
\]
Now take \(q=\mathcal N(p)/\mathcal D(p)\).  Since
\(\mathcal P(p)=0\) and \(\mathcal D(p)>0\), the first identity gives
\(\mathcal G=0\).  The second then gives \(\mathcal H=0\), because
\(cp-s<0\).  The displayed formula for \(\lambda\) makes
\(F_1=0\), hence \(\mathcal A=0\); the last identity and \(s>0\) give
\(\mathcal B=0\).  Finally \(\Delta>0\), so
\[
                         F_1=F_2=F_3=0.
\]
Thus all three equations in \eqref{eq:calibration-system} are recovered
without a hidden resultant or a numerical fit.  The three expansions in
\eqref{eq:reconstruction-chain} are also kernel-checked in
\code{CalibrationExistence.lean}.
In terms of the
golden ratio \(\varphi\), the defining equation is
\[
 (7{+}15\varphi)(p^4{+}1)+(94{+}150\varphi)(p^3{+}p)
 +(134{+}214\varphi)p^2
 =2\sin72^\circ\,(p^2{-}1)(p^2{-}4p{+}1).
\]

\subsection{The displayed placement}

For completeness, the rigid motion used in Figure~\ref{fig:path} is
\[
\begin{aligned}
 t_y&=x\cos b+y\sin b+s,\\
 t_x&=\frac{sc-(c\cos b-s\sin b)K_x-ct_y}{s},\\
 \Psi(P)&=R_{\pi/2+b}P+(t_x,t_y).
\end{aligned}                                                    \tag{B.9}\label{eq:placement}
\]
If
\[
 (P_0,\ldots,P_7)=\Psi(\sigma E,\sigma T_2,\sigma T_1,\sigma K,
                              K,T_1,T_2,E),
\]
then
\[
 (P_0)_y=(P_1)_y=0,\qquad
 -s(P_7)_x+c(P_7)_y=sc,\qquad
 s(P_4)_x+c(P_4)_y=sc.                                        \tag{B.10}\label{eq:placed-contacts}
\]
The two base equalities and the right-arm equality follow directly from
\eqref{eq:placement} and \(\sin^2b+\cos^2b=1\).  After the same
substitution, the left-arm equality is \(\cos b\) times the second
equation of \eqref{eq:contacts} minus \(\sin b\) times the first.
These placement identities are also checked in
\code{Construction.lean}.
Thus the initial segment lies on the base, the endpoint lies on the left
side, and \(\Psi(K)\) touches the right side.

\begin{corollary}[transcendence of the escape constant]
\label{cor:transcendence}
The exact escape constant \(C\) is transcendental.
\end{corollary}

\begin{proof}
The numbers \(s,c,p,q\) are algebraic.  The half-angle identities express
\(\sin a,\cos a,\sin b,\cos b\), and therefore \(\lambda\), algebraically
in terms of them.  Put \(\delta=(b-a)/2\).  Lemma~\ref{lem:length} gives
\(a<b<0\), and hence \(p<q<0\).  Thus
\[
              \tan\delta=\frac{q-p}{1+pq}
\]
is a nonzero algebraic number; here \(1+pq>1\).  If \(\delta\) were
algebraic, then \(2i\delta\) would be a nonzero algebraic number, while
\[
              e^{2i\delta}
                 =\frac{1+i\tan\delta}{1-i\tan\delta}
\]
would be algebraic.  The denominator is nonzero because
\(\tan\delta\) is real.  This contradicts Lindemann--Weierstrass
\cite{baker1975}.  Therefore \(\delta\) is transcendental.  Finally,
\[
                         C=4s\delta+2sc\lambda.
\]
Since \(s\ne0\) and \(s,c,\lambda\) are algebraic, algebraicity of \(C\)
would force \(\delta\) to be algebraic.
\end{proof}

\subsection{What is verified, and how}\label{sec:boundary}

The argument has exactly two finite certificate families.  They are the
only places where a long exact calculation is delegated to machine
checking; neither is numerical evidence.

\begin{enumerate}
\item \textbf{Calibration.}  This certificate checks the isolated root
  of \eqref{eq:calibration-quartic} via
  \eqref{eq:quintic-deflation}--\eqref{eq:quintic-isolation}; reconstructs
  \eqref{eq:calibration-system} via
  \eqref{eq:denominator-positive}--\eqref{eq:reconstruction-chain}; and
  checks the coarse scalar bounds used in
  \eqref{eq:strict-ledger-states}, the side-sector-mass bound in
  Lemma~\ref{lem:ledger}, and folded-fan separation.  It also checks
  the length identity in its algebraic form: the seven-piece total
  \(2\{K_x+\tau+s(b-a)+d\}\) equals \(C\) exactly.  Reading that total
  as \(\len\Gamma\)---that is, summing the piece lengths of
  \eqref{eq:piece-table}---is elementary and remains prose.
\item \textbf{Geometry.}  The seven-piece curve, its junctions and
  piece lengths, the strict \(>3\) determinant bound for its contact
  system, and, after the analytic support-fan identification of
  Lemma~\ref{lem:candidate-fan}, every strict or exact entry in the
  eighteen windows of Proposition~\ref{prop:attainment}.
\end{enumerate}

Both are checked in Lean~4 over the rationals: every numerical
inequality used in the proof reduces to exact rational comparisons between
certified interval endpoints.  The accompanying \code{lean-toolchain}
and \code{lake-manifest.json} pin the compiler and the exact Mathlib
revision, and a single \code{lake build} checks the full dependency
closure.  The development, the independent Python audits, and the paper
source are available at \url{https://github.com/atemerev/gnomon}.
The development contains no \code{sorry}, no
\code{native\_decide}, and no added axiom.  The dedicated module
\code{AxiomAudit.lean} prints the dependencies of every exported
certificate theorem used here; each is a subset of Mathlib's standard
classical trio of \code{propext}, \code{Classical.choice}, and
\code{Quot.sound}, and the build fails if any dependency were to fall
outside that trio.
The outer-order exclusion is now the rational support calculation
\eqref{eq:rational-supports}--\eqref{eq:rational-support-bounds}, not a
computational input.  For independent cross-checking,
\code{verify\_rational\_supports.py} exactly re-expands the eight
printed supported sums, and \code{OuterTetral.lean} retains the
original fixed certificates for all fourteen orders.

The reusable part of the argument is also formalized, and is independent
of both certificate families: the escape identity
\eqref{eq:balance}, the first-segment deduction underlying
Lemma~\ref{lem:bitonic}---granting the interleaving criterion for chords
of a convex polygon, which is its only geometric input---the normal-cone
aggregation of
Lemma~\ref{lem:normal-aggregation} in the finite form actually used, the
strictly shortening uncrossing surgery of
Proposition~\ref{prop:standard}, the Abel summation of
Lemma~\ref{lem:calibration},
the exact circle states \eqref{eq:unit-ledger-blocks}, the unit-junction
chain \eqref{eq:unit-ledger-junctions}, the two exact interior-state
formulas \eqref{eq:strict-ledger-states}, the radial identities
\eqref{eq:radial-derivatives}, the ledger-shape deduction of
Lemma~\ref{lem:sweep}, and the metric core of
Lemma~\ref{lem:local-gap}: the doubled-angle identity for the reflected
chord, the wedge bound \(\theta<\pi/4\) of step (ii), the implication
\(h\le d\), and the reflection estimate (A.6) itself.  What remains
outside Lean in that lemma is its planar combinatorics: the exclusion of
alternative \({\rm(L)}\) and the hull-preserving surgery of
Lemma~\ref{lem:outer-caps}.

We have also reproduced the construction by a route that avoids the
quartic entirely.  Solving the contact and stationarity system by
high-precision Newton iteration, followed by integer relation detection,
recovers the same parameters and the same critical length to more than
twenty-five digits; and a rational line--arc witness, certified in exact
arithmetic over \(\Q(\sqrt5,\sqrt{10-2\sqrt5})\), gives the independent
bound \(\mathcal E(G)<1.282680\), which exceeds \(C\) by about
\(3.8\times10^{-6}\).  Neither computation is used below; both are
recorded only as cross-checks on the data of
Section~\ref{sec:construction}.

The only imported geometric result is the \(\Lambda\)-configuration
theorem, Theorem~\ref{thm:lambda}; Proposition~\ref{prop:standard}
supplies exactly its hypotheses.

The remaining steps are proved here in conventional prose, and each is
planar geometry rather than computation: the compactness argument of
Proposition~\ref{prop:standard}; the one-turn support principle of
Lemma~\ref{lem:one-turn-support} and the candidate normal-list
identification of Lemma~\ref{lem:candidate-fan}; the planar incidence and
hull surgery of Lemmas~\ref{lem:outer-caps} and~\ref{lem:local-gap},
including the exclusion of alternative \({\rm(L)}\); endpoint peeling in
Lemma~\ref{lem:bitonic}; the cutting of the source measure into finitely
many normal-fan cells preceding Lemma~\ref{lem:normal-aggregation}; and
the folded-fan identification of Lemma~\ref{lem:fan-ledger}.

\end{document}